\documentclass[11pt]{article}

\usepackage[a4paper]{geometry}
\usepackage{a4wide}

\usepackage{amsmath}
\usepackage{amsfonts}
\usepackage{amsthm}
\usepackage{amssymb}
\usepackage{amsopn}
\usepackage{graphicx}
\usepackage{epstopdf}
\usepackage{algorithmic}
\usepackage{soul,xcolor}
\usepackage{tikz}
\ifpdf
  \DeclareGraphicsExtensions{.eps,.pdf,.png,.jpg}
\else
  \DeclareGraphicsExtensions{.eps}
\fi

\newcommand*{\argmin}[1]{{\underset{#1}{\operatorname{argmin}}}}
\newcommand*{\nm}[1]{{\left\|#1\right \|}}
\newcommand*{\nmu}[1]{{\|#1\|}}
\newcommand*{\bbC}{\mathbb{C}}
\newcommand*{\bbI}{\mathbb{I}}
\newcommand*{\bbR}{\mathbb{R}}
\newcommand*{\bbZ}{\mathbb{Z}}
\newcommand*{\cT}{\mathcal{T}}
\newcommand*{\rH}{\mathrm{H}}
\def\D{\,\mathrm{d}}
\def\I{\mathrm{i}}
\def\E{\mathrm{e}}

\newtheorem{theorem}{Theorem}
\newtheorem{lemma}[theorem]{Lemma}
\newtheorem{remark}[theorem]{Remark}

\title{Stable and accurate least squares radial basis function approximations on bounded domains}

\author{Ben Adcock, Daan Huybrechs and C{\'e}cile Piret}

\begin{document}

\maketitle

\begin{abstract}
The computation of global radial basis function (RBF) approximations requires the solution of a linear system which, depending on the choice of RBF parameters, may be ill-conditioned. We study the stability and accuracy of approximation methods using the Gaussian RBF in all scaling regimes of the associated shape parameter. The approximation is based on discrete least squares with function samples on a bounded domain, using RBF centers both inside and outside the domain. This results in a rectangular linear system. We show for one-dimensional approximations that linear scaling of the shape parameter with the degrees of freedom is optimal, resulting in constant overlap between neighbouring RBF's regardless of their number, and we propose an explicit suitable choice of the proportionality constant. We show numerically that highly accurate approximations to smooth functions can also be obtained on bounded domains in several dimensions, using a linear scaling with the degrees of freedom per dimension. We extend the least squares approach to a collocation-based method for the solution of elliptic boundary value problems and illustrate that the combination of centers outside the domain, oversampling and optimal scaling can result in accuracy close to machine precision in spite of having to solve very ill-conditioned linear systems.
\end{abstract}

\section{Introduction}

\subsection{Radial basis function approximations}\label{ssec:rbf_approx}

Radial basis functions provide a convenient way to approximate functions in one or more dimensions based on translates of a single function $\phi(r)$. The argument $r = \Vert x - \xi \Vert$ is the radial (Euclidean) distance between an \emph{evaluation point} $x$ and a \emph{center} $\xi$. The function $\phi$ itself is called the radial basis function (RBF) (or radial function). By covering a domain with centers, functions can be approximated using linear combinations of the associated centered RBFs and this approximation is traditionally obtained by interpolation. A wide variety of radial basis functions exist with very different properties: for example, $\phi(r)$ may be smooth or non-smooth at $r=0$, and it may be decaying or increasing for $r \to \infty$. The most commonly used smooth radial functions are listed in Table \ref{RadialFunctions}. Note that they all admit a \emph{shape parameter} $\varepsilon$ that controls their shape. This leads to an approximation of the form
\begin{equation}\label{RBF_direct}
 f(x) \approx \sum_{n=1}^N \lambda_n \phi(\varepsilon(x - \xi_n)).
\end{equation}
Small values of the shape parameter dilate the RBF: the limit $\varepsilon \to 0$ is often referred to as the \emph{flat limit}. The choice of an appropriate shape parameter has significant consequences for both accuracy and stability of corresponding numerical methods. Most crucial is to decide how to adjust the value of $\varepsilon$ when the number of centers, $N$, grows. One can either keep $\varepsilon$ fixed, or make it grow with $N$, \cite{Platte2011}. In this study, we will focus on sublinear and linear scalings of $\varepsilon$ with $N$. That is, we will consider respectively $\varepsilon = cN^{\alpha}$, for $0<\alpha<1$, and $\varepsilon = cN$, for some positive constant $c$.

\begin{table}[!ht]
\centering
\begin{tabular}{ccc}
\hline
\textbf{Name of RBF}  & \textbf{Abbreviation}  & \textbf{Definition}\tabularnewline
\hline
Gaussian  & GA  & $e^{-(\varepsilon r)^{2}}$\tabularnewline
Multiquadric  & MQ  & $\sqrt{1+(\varepsilon r)^{2}}$\tabularnewline
Inverse Quadratic  & IQ  & $\frac{1}{1+(\varepsilon r)^{2}}$\tabularnewline
Inverse Multiquadric  & IMQ  & $\frac{1}{\sqrt{1+(\varepsilon r)^{2}}}$\tabularnewline
\hline 
\end{tabular}\caption{Definitions of some common radial functions $\phi(\varepsilon r)$.}\label{RadialFunctions}
\end{table}

\subsection{Existing results and analysis}\label{ssec:rbf_existing}

Generally speaking, smaller values of $\varepsilon$ lead to higher approximation accuracy, but worse condition numbers. Larger values of $\varepsilon$ have better conditioning, but yield less accuracy or even, in the case of bounded domains, no accuracy at all. The trade-off between small and large shape parameters was referred to in \cite{Schaback95} as an \emph{uncertainty principle}. However, this tradeoff was shown to originate not from the interpolation problem, which is well posed, but from numerically solving the problem via the standard representation (\ref{RBF_direct}), having to multiply large oscillatory expansion coefficients $\lambda_n$ with flat radial function translates $\phi(\varepsilon(x - \xi_n))$, in essence a sort of removable singularity \cite{Driscoll_flat_limit,RBF_book_Fornberg_Flyer}.

This has spurred a quest for algorithms to analytically remove the ill-conditioning of linear systems associated with small shape parameter values. For instance, the Contour-Pad\'e method extends $\varepsilon$ to the complex plane and makes use of numerical contour integration to stably compute the RBF interpolant \cite{FORNBERGPade}. The RBF-QR method induces a change of bases via a QR factorization, exposing the problematic terms and removing them (and with them the ill-conditioning) altogether \cite{RBFQR1, RBFQR2, RBFQR_GA_R}. Other such algorithms have been introduced, since keeping a small fixed shape parameter has often been the preferred strategy (see \cite{RBF_RA, FasshauerQR1, FasshauerQR2, RBF_book_Fornberg_Flyer}).

Indeed, it is observed that the linear scaling regime mentioned above often yields unstable methods, \cite{BOYD2010}; however some strategies have been introduced to attain accurate approximations in this regime \cite{Mazya,PiretFrames}. The remaining error is known as the stagnation or saturation error.
%Without stable algorithms, the delicate task of adjusting the shape parameter with $N$ to reduce the error while preserving stability is crucial. 
The lack of accuracy in the linear scaling regime can be avoided by using slower than linear growth such as $\varepsilon = c \sqrt{N}$, although the procedure is still unstable \cite{platte2005polynomials,Platte2011}. Alternatively, the ill-conditioning of global RBF approximation systems can be avoided altogether by considering local approximations of small size only, for example as part of the computation of a local stencil for a finite difference method. Finally, the approximation accuracy can be improved by augmenting the set of centered RBFs with other basis functions, such as polynomials~\cite{fasshauer2007meshfree}.

\subsection{Contribution of the paper}
We describe and analyze a numerical method which in the linear scaling regime $\varepsilon \sim cN$ is numerically stable for approximation on bounded domains. We fully quantify the setting for the specific choice of the Gaussian RBF (see \eqref{eq:gaussian} below), which is smooth and rapidly decaying as $r \to \infty$. The setting crucially relies on the combination of the following points:
\begin{itemize}
    \item accuracy in the linear scaling regime on bounded domains is improved by adding centers outside the domain,
    \item numerical stability of the computations is achieved by considering discrete least squares approximation instead of interpolation, i.e., by using more sample points than centers,
    \item in addition the sampling points may in general be different from the centers.
\end{itemize}
These points were inspired by the analysis of function approximation in the presence of redundancy using frames in \cite{fna1,fna2}.

These points have appeared in the literature before in the context of radial basis functions. Using least squares approximation on bounded domains with centers outside the domain was recently explored in \cite{frykland2018,hashemi2021}. Compared to these references, we analyze the (sub)linear scalings, provide justification for the scheme and identify explicit choices of the parameters involved.

\subsection{Methodology and results}
It follows from the analysis in \cite{fna1,fna2} that least squares approximations strike a balance between the approximation error and the size of the coefficients $\Vert \lambda \Vert$ (see \eqref{eq:ENeps} for the corresponding error bound). Thus, in order to analyze the accuracy of an approximation, one has to analyze these quantities simultaneously. The main theoretical results of the paper are formulated in Theorems~\ref{thm:accuracylimit} and \ref{thm:convergence}. The results are stated in terms of a user-chosen desired accuracy parameter $\tau \ll 1$, which is used to regularize the ill-conditioned linear system.

Theorem \ref{thm:accuracylimit} describes the limiting accuracy for $N \to \infty$. It is shown that the error may converge to zero in the sublinear scaling regime, while a linear scaling regime $\varepsilon \sim c N$ leads to a saturation error that depends on the choice of $c$. However, since the result is explicit, $c$ can be chosen to make the saturation error arbitrarily small, e.g., on the order of the small parameter $\tau$.

Theorem \ref{thm:convergence} describes the convergence rate for smooth functions. For a sublinear scaling regime, improving the approximation error comes at a cost of a rapidly increasing coefficient norm. Hence, as a consequence of ill-conditioning, the best approximations are not numerically computable. Taking both approximation error and coefficient norm into account, Theorem \ref{thm:accuracylimit} shows slower convergence in the sublinear scaling regime compared with the linear scaling regime. For the case of Gaussian RBF's, an optimal value of $c$ in the scaling $\varepsilon = c N$ is given explicitly by \eqref{eq:optimal_c}.

Broadly speaking, the methodology of the proof relies on the two points raised above in the following way. The centers outside the domain ensure that the saturation error can be made arbitrarily small. The oversampling ensures that good approximations with bounded coefficient norm can be found in practice, using function samples only and using inexact floating point computations.

\subsection{Overview of the paper}
We introduce the LS-RBF method in \S\ref{sec:ls-rbf}. The relevant results of \cite{fna1,fna2} are reviewed in \S\ref{sec:redundancy}. The main theoretical results for univariate problems using the Gaussian RBF are formulated in Theorems~\ref{thm:accuracylimit} and \ref{thm:convergence}. The theory is illustrated with a range of experiments in \S\ref{sec:experiments}.

\section{The LS-RBF method}
\label{sec:ls-rbf}

We now introduce the \textit{Least Squares Radial Basis Function (LS-RBF)} method. We distinguish between the centers $\xi_n$ of the RBF's, $n=1,\ldots,N$, and the sampling points $x_m \in \Omega$, $m=1,\ldots,M$. In the least squares approach, $M > N$. The sampling points belong to a domain $\Omega \subset \mathbb{R}^d$ on which the function $f$ to be approximated is defined. The centers may also be outside the domain.

The approximation to $f$ is
\begin{equation}
\label{eq:rbf}
\tilde f = \sum^{N}_{n=1} \lambda_n \phi(\varepsilon(\cdot - \xi_n)) = \sum^{N}_{n=1} \lambda_n \phi_n,
\end{equation}
where $\lambda = (\lambda_n)^{N}_{n=1}$ is given by the discrete least squares solution
\begin{equation}
\label{eq:leastsquaresproblem}
\lambda = \argmin{z \in \mathbb{R}^N} \| A z - b \|_{\ell^2},   
\end{equation}
with
\begin{equation}
\label{eq:matrix_and_rhs}
A = (\phi_n(x_m))^{M,N}_{m,n=1} \in \mathbb{R}^{M \times N},\qquad b = (f(x_m))^{M}_{m=1} \in \mathbb{R}^M.
\end{equation}
In our analysis we focus on the Gaussian RBF
\begin{equation}\label{eq:gaussian}
    \phi(r) = \E^{-r^2},
\end{equation}
but in the numerical experiments we also consider other RBF's.

Depending on the distribution of the centers, the samples and the shape parameter, the matrix $A$ may be ill-conditioned. In that case, we solve $Ax=b$ using a direct solver with regularization threshold $\tau$. In this paper we analyze the use of a truncated singular value decomposition, in which all singular values below $\tau$ are discarded. In practice, similar results are obtained using a rank-revealing QR decomposition with pivoting.

The solution using regularization with threshold $\tau$ is denoted by $\lambda^{(\tau)}$, in contrast with the exact solution $\lambda$ of the least squares problem \eqref{eq:leastsquaresproblem}. However, in the following we will frequently omit the superscript.

\begin{remark}\label{rem:normalization}
In the analysis further on we have normalized the radial basis functions, as well as the matrix $A$. We have applied these normalizations in our implementation too, hence we describe them here. Thus, in theory and in practice, we (i) replace $\phi(\epsilon(\cdot - \xi_n))$ in \eqref{eq:rbf} by $\sqrt{\epsilon} \phi(\epsilon(\cdot - \xi_n))$, and we (ii) multiply both $A$ and $b$ by $\sqrt{\frac{T}{M}}$. Normalization of the RBF's affects the size of $\| \lambda \|$, the relevance of which is motivated in \S\ref{sec:redundancy}. The normalization of $A$ and $b$ also affects our results because we use an absolute threshold $\tau$ in the regularization. The rationale of this scaling and its possible generalizations are elaborated on in Remark~\ref{rem:constants}.
\end{remark}

\section{Numerical approximations with redundancy}\label{sec:redundancy}

The ill-conditioning of the least squares matrix $A$ indicates that the RBF's may be close to linearly dependent, or in other words, nearly redundant.
The approximation of functions using redundancy was extensively analyzed in \cite{fna1,fna2}. We briefly review the main results, which will guide our analysis.

Define the \emph{synthesis operator} $\mathcal{T}_N : \mathbb{C}^N \rightarrow \mathrm{L^2(\Omega)}$ by
\begin{equation*}
    \mathcal{T}_N \lambda = \sum^{N}_{n=1} \lambda_n \phi_n.
\end{equation*}
That is, $\mathcal{T}_N$ associates with a set of coefficients the corresponding expansion in the RBF basis, which was defined in \eqref{eq:rbf}.

The solution vector $\lambda^{(\tau)}$ of the linear system $Ax=b$, with $A$ and $b$ given by \eqref{eq:matrix_and_rhs}, and using SVD regularization with threshold $\tau$, satisfies \cite[Theorem 1.3]{fna2}
\begin{equation}
    \label{eq:fna2_error}
 \Vert f - \mathcal{T}_N \lambda^{(\tau)} \Vert_{L^2(\Omega)} \leq \inf_{z \in \bbC^N} \{\Vert f - {\mathcal T}_N z \Vert_{L^2(\Omega)} + c_{M,N}^\tau \Vert f - \mathcal{T}_N z \Vert_{M}  + \tau \, d_{M,N}^\tau \Vert z \Vert_{l^2} \}. 
\end{equation}
Here, $c_{M,N}^\tau$ and $d_{M,N}^\tau$ are constants that depend on the other parameters of the RBF approximation problem ($M$, $N$ and $\tau$) and in addition on the specific choice of sampling points and centers. The quantity $\Vert \cdot \Vert_M$ is a discrete norm defined over the sample points $x_m$.  In particular, if the sample points $x_m$ fill the domain $\Omega$ quasi-uniformly as $M \rightarrow \infty$, one has that $c^{\tau}_{M,N},d^{\tau}_{M,N} \lesssim 1$ and $\nm{f - \cT_N z}_{M} \lesssim \nm{f - \cT_N z}_{L^2(\Omega)}$ for sufficiently large $M$.  Hence, under such conditions, the error of the least-squares approximation is determined by the quantity
\begin{equation}
\label{eq:ENeps}
E_{N,\tau}(f) = \inf_{z \in \bbC^N} \left \{ \nm{f - \mathcal{T}_N z }_{L^2(\Omega)} + \tau \nm{z}_{\ell^2} \right \}.
\end{equation}

Importantly, this term indicates that an accurate RBF approximation will be found if an accurate approximation exists with a coefficient vector $z$ with small discrete norm $\Vert z \Vert$. The numerically computed solution will be as accurate as any such vector $z$, regardless of the condition number of $A$, up to the threshold $\tau$.

In view of the generic error estimate \eqref{eq:ENeps}, in our analysis it is sufficient to show the existence of accurate RBF approximations with corresponding bounds on the sizes of their expansion coefficients.

\begin{remark}\label{rem:constants}
We elaborate briefly on the behaviour of the constants $c_{M,N}^\tau$ and $d_{M,N}^\tau$ in \eqref{eq:fna2_error} for increasing $M$, and their relation to the choice of sampling points, in order to better motivate the simpler form~\eqref{eq:ENeps}. It is shown in \cite[Proposition 3.10]{fna2} that these constants satisfy $c_{M,N}^\tau,d_{M,N}^\tau \leq \frac{1}{\sqrt{A'_{M,N}}}$ where
\[
 A'_{M,N} = \inf_{\substack{g \in \rH_N \\ \| g \|=1}} \| g \|^2_M,
\]
with in the setting of this paper $H_N = \mathrm{span}\{\phi_n\}_{n=1}^N$ and
\[
\| g \|^2_M = \frac{1}{M} \sum_{m=1}^M g(x_m) \overline{g(x_m)}.
\]
Ideally, one ensures that the discrete and continuous norms agree for large $M$:
\begin{equation}\label{eq:riemannsumcondition}
\| g \|^2_M \to \|g \|_{L^2(\Omega)}^2 \quad \mbox{or} \quad \frac{1}{M} \sum_{m=1}^M g(x_m) \overline{g(x_m)} \to \int_\Omega g(x) \overline{g(x)} {\rm d}x, \qquad M \to \infty.
\end{equation}
In that case both constants $c_{M,N}^\tau$ and $d_{M,N}^\tau$ simply tend to $1$. This is the case for all examples in this paper because we use equispaced points and \eqref{eq:riemannsumcondition} becomes a Riemann sum. For non-equispaced points, one may consider weighting each sampling point $x_m$ with a weight $\sqrt{w_m}$. In that case $\| g \|^2_M = \sum_{m=1}^M w_m g(x_m) \overline{g(x_m)}$, and a proper choice of weights can ensure that a weighted analogue of \eqref{eq:riemannsumcondition} holds.
\end{remark}

\begin{remark}
The previous remark pertains to the oversampling limit $M \to \infty$. For fixed $M$ and $N$, \cite[Proposition 3.10]{fna2} also shows that $c_{M,N}^\tau, d_{M,N}^\tau \lesssim \frac{1}{\varepsilon}$. For $M$ close to $N$, this bound may be sharp and both constants may be large. Thus, without `sufficient' oversampling \eqref{eq:fna2_error} does not demonstrate accuracy of the approximation. In particular, there is no reason to assume that an RBF interpolant ($M=N$), rather than a least squares approximation ($M > N$), would yield an accurate function approximation in general. We will show with numerical experiments that $M \sim \gamma N$ for some $\gamma >1$ appears to be sufficient in our setting and we generally choose $\gamma=2$. It remains an open problem to prove that such linear oversampling is indeed sufficient.
\end{remark}

\section{Analysis in the one-dimensional case for Gaussian RBFs}

In this section, we study the behaviour of \eqref{eq:ENeps} for the Gaussian RBF \eqref{eq:gaussian}. For simplicity, we consider the one-dimensional case, where $\Omega \subset (-T,T)$ is the domain.  Our analysis relies crucially on the radius of $\Omega$, defined as
\begin{equation}
\label{bOmegadef}
B = \max_{x \in \Omega} |x|,
\end{equation}
Note that $B < T$ because $\Omega$ is a subset of $[-T,T]$. 
We also consider equally-spaced centers on $[-T,T]$ given by
\begin{equation*}
\xi_n = \frac{n T}{N},\quad n = -N,\ldots,N.
\end{equation*}
For convenience, we assume the approximation uses these $2N+1$ centers (as opposed to $N$ in \eqref{eq:rbf}), indexed from $-N$ to $N$.  That is,
\begin{equation}\label{rbf_sum_equispaced}
f \approx \cT_N \lambda = \sum^{N}_{n=-N} \lambda_n \sqrt{\varepsilon} \phi(\varepsilon(\cdot - n T/N)).
\end{equation}
Note that we now also normalize the RBF functions by a factor $\sqrt{\varepsilon}$, so that the $L^2$-norms of the corresponding terms in the sum are independent of $\varepsilon$.

We define the Fourier transform of $f \in L^2(\bbR)$ by
\begin{equation}
\hat{f}(\omega) = \int^{\infty}_{-\infty} f(x) \E^{-\I \omega x} \D x,\qquad \omega \in \bbR,
\end{equation}
with inverse
\begin{equation}
    \check{g}(x) = \frac{1}{2\pi} \int^{\infty}_{-\infty} g(\omega) \E^{\I \omega x} \D \omega,\qquad x \in \bbR,
\end{equation}
for $g \in L^2(\bbR)$.  We also recall Plancherel's formula
\begin{equation*}
\Vert \hat{f} \Vert_{L^2(\mathbb{R})}  = \sqrt{2 \pi} \Vert f \Vert_{L^2(\mathbb{R})}.
\end{equation*}

\subsection{Preliminary results}

Our main results are contained in \S \ref{ss:limiting} and \S \ref{ss:smooth}.  First, in this subsection, we state and prove a number of preliminary results.

\begin{lemma}
\label{lem:any_g_and_any_lambda}
Let $f \in L^2(\Omega)$, let $g \in L^2(\bbR)$ be any extension of $f$ to the real line, and let $\lambda \in \ell^2(\bbZ)$ be arbitrary.  Then
\begin{equation*}
E_{N,\tau}(f) \leq \nm{g - h}_{L^2(\bbR)}+ \left ( \sqrt{\frac{\sqrt{2 \pi} B N}{T} \mathrm{erfc}(\sqrt{2} \varepsilon(T-B))} + \tau \right ) \nm{\lambda}_{\ell^2},
\end{equation*}
where $E_{N,\tau}(f)$ is as in~\eqref{eq:ENeps}, $h = \sum^{\infty}_{n=-\infty} \lambda_n \sqrt{\varepsilon} \phi(\varepsilon(\cdot - n T/N) ) \in L^2(\bbR)$ and $\mathrm{erfc}$ is the complementary error function. In particular, if
\begin{equation}
\label{eps_min_simplify}
\varepsilon \geq  \frac{1}{\sqrt{2}(T-B)} \sqrt{\log\left (\frac{\sqrt{2 \pi} B N}{T \tau^2} \right ) } ,
\end{equation}
then
\begin{equation*}
E_{N,\tau}(f) \leq \nm{g - h}_{L^2(\bbR)} + 2 \tau \nm{\lambda}_{\ell^2}.
\end{equation*}
\end{lemma}

The motivation behind this lemma is the following.  Analyzing $E_{N,\tau}(f)$ directly requires constructing a finite series $\sum^{N}_{n=-N} \lambda_n \phi_n$ which approximates $f$ over $\Omega$ and whose coefficients $\lambda_n$ do not grow too large.  This can be difficult, due to the requirement that the series be finite and the fact that the domain $\Omega$ is not the whole real line.  Instead, this lemma states that it is sufficient to study how well any \emph{extension} $g \in L^2(\mathbb{R})$ of $f$ can be approximated over the whole real line via an \emph{infinite} series whose terms are not too large.  As we shall see later, by moving to infinite series on the real line we are able to employ Fourier analysis techniques.

\begin{proof}[Proof of Lemma~\ref{lem:any_g_and_any_lambda}]
We first show that the function $h \in L^2(\bbR)$.
Let $y = N x / T$ and $\psi(y) = \phi(\varepsilon T y / N)$.  Then $h \in L^2(\bbR)$ if and only if the function $\sum^{\infty}_{n=-\infty} \lambda_n \psi(\cdot - n) \in L^2(\bbR)$ for $\lambda \in \ell^2(\bbZ)$.  By Lemma 9.2.2 and Theorem 9.2.5 of \cite{christensen2016introduction}, this holds provided
\begin{equation}
\label{christensentest}
\sum_{k \in \bbZ} | \hat{\psi}(2 \pi (\omega + k)) |^2 \leq C,\quad a.e.\ \omega \in [0,1], 
\end{equation}
for some constant $C > 0$.
Observe that
\begin{equation*}
\sum_{k \in \bbZ} | \hat{\psi}(2 \pi (\omega + k)) |^2  = \frac{N^2}{\varepsilon^2 T^2} \sum_{k \in \bbZ} | \hat{\phi}(2 \pi N (\omega+k) / (\varepsilon T) ) |^2.    
\end{equation*}
Since $\phi$ is the Gaussian, we have
\begin{equation*}
\hat{\phi}(\omega) = \sqrt{\pi} \E^{-\omega^2/4}.
\end{equation*}
and therefore
\begin{align*}
\sum_{k \in \bbZ} | \hat{\psi}(2 \pi (\omega + k)) |^2 = \frac{N^2 \pi }{\varepsilon^2 T^2} \sum_{k \in \bbZ} \exp(-2 \pi^2 N^2 (\omega+k)^2/(\varepsilon^2 T^2) ).
\end{align*}
This sum is clearly uniformly bounded in $\omega$.  Hence \eqref{christensentest} holds for some suitable $C$.  We conclude that $h \in L^2(\mathbb{R})$.

With this in hand, we let $z = (\lambda_n)^{N}_{n=-N}$ in \eqref{eq:ENeps} to obtain
\begin{align*}
E_{N,\tau}(f) &\leq \nm{f - h}_{L^2(\Omega)} + \nm{h - \cT_N z}_{L^2(\Omega)} + \tau \nm{\lambda}_{\ell^2}
\\
& \leq \nm{g - h}_{L^2(\bbR)} + \sum_{|n| > N} |\lambda_n| \sqrt{\varepsilon}\nm{\phi(\varepsilon(\cdot - n T/N))}_{L^2(\Omega)} + \tau \nm{\lambda}_{\ell^2}.
\end{align*}
Consider the middle term.  By \eqref{bOmegadef}, we have
\begin{align*}
 \nm{\phi(\varepsilon(\cdot - n T/N))}_{L^2(\Omega)} \leq \sqrt{2 B} \E^{-\varepsilon^2(B - |n| T/N)^2}.
\end{align*}
Therefore
\begin{align*}
\sum_{|n| > N} |\lambda_n| \sqrt{\varepsilon}\nm{\phi(\varepsilon(\cdot - n T/N))}_{L^2(\Omega)} & \leq \sqrt{4 B} \nm{\lambda}_{\ell^2} \sqrt{\varepsilon} \sqrt{\sum_{n > N} \E^{-2\varepsilon^2(n T/N - B)^2} }
\\
& \leq \sqrt{4 B} \nm{\lambda}_{\ell^2} \sqrt{\varepsilon}\sqrt{\int^{\infty}_{N} \E^{-2 \varepsilon^2 (x T /N - B)^2} \D x }
\\
& = \sqrt{4 B} \nm{\lambda}_{\ell^2} \sqrt{\varepsilon}\sqrt{\frac{N}{\sqrt{2} T \varepsilon } \int^{\infty}_{\sqrt{2} \varepsilon(T -b)} \E^{-y^2} \D y }
\\
& = \sqrt{\frac{\sqrt{2 \pi} B N}{T} \mathrm{erfc}(\sqrt{2} \varepsilon(T-B))} \nm{\lambda}_{\ell^2}. 
\end{align*}
This completes the proof of the first result.

For the second result, observe that $\mathrm{erfc}(z) \leq \E^{-z^2}$.  This follows from \cite[Eqn.\ (7.8.2)]{NIST:DLMF}, noting that $1/(x+\sqrt{x^2+4/\pi}) < 1$ for $x > 0$. Hence, due to condition~\eqref{eps_min_simplify} on $\varepsilon$,
\begin{equation*}
\frac{\sqrt{2 \pi} B N}{T} \mathrm{erfc}(\sqrt{2} \varepsilon(T-B)) \leq \frac{\sqrt{2 \pi} B N}{T} \exp(-2 \varepsilon^2(T-B)^2) \leq \tau^2.     
\end{equation*}
The result now follows immediately from the first part.
\end{proof}

This lemma states that when $\varepsilon$ grows mildly with $N$, we reduce the problem of estimating $E_{N,\tau}(f)$ to that of studying the approximation of a function $g \in L^2(\bbR)$ via the corresponding infinite expansion $h$ in the RBF system.  Note that the growth condition \eqref{eps_min_simplify} is primarily for convenience and it is not a fundamental restriction in the analysis: it merely allows us to simplify the bulky expression in the first part of the lemma, which would otherwise propagate into all bounds later on. The condition will be satisfied in all examples considered later.

With this in hand, to estimate $E_{N,\tau}(f)$ we now construct a suitable function $h$ via Fourier analysis.  This is the topic of the following two lemmas.  Specifically, in view of Plancherel's identity, we pick $h$ so that $\hat{h}$ agrees with $\hat{f}$ for all low frequencies.

\begin{lemma}
\label{lem:ENtau_main_bd_general}
Let $f \in L^2(\Omega)$ and let $g \in L^2(\bbR)$ be any extension of $f$ to the real line.  If $\varepsilon$ satisfies \eqref{eps_min_simplify}
then, for any $K \leq N$,
\begin{align*}
E_{N,\tau}(f) \lesssim & \sqrt{\int_{|\omega| > K \pi/T} | \hat{g}(\omega) |^2 \D \omega} 
\\
&+ \left( \frac{1}{\sqrt{\E^{N^2 \pi^2 / (2 T^2 \varepsilon^2)}-1}}+ \sqrt{\frac{T \varepsilon}{N}} \tau \right ) \sqrt{\int_{|\omega| \leq K \pi/T} \frac{| \hat{g}(\omega) |^2 }{| \hat{\phi}(\omega/\varepsilon) |^2} \D \omega}.
\end{align*}
\end{lemma}

\begin{proof}
Consider the function $F : [-N \pi/T,N \pi/T] \rightarrow \bbC$ defined by
\begin{equation*}
F(\omega) = \frac{\sqrt{\varepsilon}\hat{g}(\omega)}{ \hat{\phi}(\omega / \varepsilon)} \bbI_{[-K \pi/T,K \pi/T]}(\omega).
\end{equation*}
Note that $F \in L^2([-N\pi/T,N \pi/T])$ since $\hat{g} \in L^2(\bbR)$ and $\hat{\phi}$ is strictly positive and bounded.  Hence $F$ has the Fourier series
\begin{equation*}
F(\cdot) =  \sum^{\infty}_{n = -\infty} \lambda_n \E^{-\I n T \cdot /N } ,
\end{equation*}
which converges in $L^2([-N \pi/T,N \pi/T])$ with coefficients $\lambda \in \ell^2(\bbZ)$ satisfying, by Parseval's identity,
\begin{equation}
\label{lambda_norm_bd}
\nm{\lambda}^2_{\ell^2} = \frac{T}{2 N \pi} \nm{F}^2_{L^2([-N \pi/T,N\pi/T])} = \frac{T\varepsilon}{2 N \pi } \int_{|\omega| \leq K \pi/T} \frac{ | \hat{g}(\omega) |^2 }{| \hat{\phi}(\omega/\varepsilon) |^2} \D \omega . 
\end{equation}
Let $h = \sum^{\infty}_{n=-\infty} \lambda_n \sqrt{\varepsilon} \phi(\varepsilon(\cdot - n T/N))$ and notice that $h \in L^2(\bbR)$ by Lemma~\ref{lem:any_g_and_any_lambda}.  Observe that
\begin{equation*}
\hat{h}(\omega) = \frac{\hat{\phi}(\omega/\varepsilon)}{\sqrt{\varepsilon}} \sum^{\infty}_{n=-\infty} \lambda_n \E^{-\I \omega n T/N}.
\end{equation*}
In particular,
\begin{equation*}
\hat{h}(\omega) = \hat{g}(\omega) \bbI_{[-K \pi/T,K \pi/T]}(\omega),\qquad |\omega | \leq N \pi/T.
\end{equation*}
Hence
\begin{align*}
\nm{g - h}_{L^2(\bbR)} &= \frac{1}{\sqrt{2 \pi}} \nmu{\hat{g} - \hat{h}}_{L^2(\bbR)} 
\\
& = \frac{1}{\sqrt{2 \pi}} \sqrt{\int_{|\omega| > K \pi/T} | \hat{g}(\omega) - \hat{h}(\omega) |^2 \D \omega } .
\\
& \leq \frac{1}{\sqrt{2 \pi}} \left( \sqrt{\int_{|\omega| > K \pi/T} | \hat{g}(\omega) |^2 \D \omega} + \sqrt{\int_{|\omega| > N \pi/T} | \hat{h}(\omega) |^2 \D \omega} \right).
\end{align*}
Consider the second term.  We have
\begin{align*}
\int_{|\omega| > N \pi/T} | \hat{h}(\omega) |^2 \D \omega &= \sum_{l \neq 0} \int^{(2l+1) N \pi/T}_{(2l-1) N \pi/T} \varepsilon^{-1} | \hat{\phi}(\omega/\varepsilon) |^2 | F(\omega) |^2 \D \omega
\\
& \leq 2 \pi \nm{F}^2_{L^2([-N \pi/T,N\pi/T])} \sum^{\infty}_{l=1} \varepsilon^{-1} \E^{-(2l-1)^2 N^2 \pi^2 / (2 T^2 \varepsilon^2)} .
\end{align*}
Replacing $(2l-1)^2$ by $l$ and summing the resulting geometric series gives
\begin{equation*}
\int_{|\omega| > N \pi/T} | \hat{h}(\omega) |^2 \D \omega \lesssim \varepsilon^{-1} \nm{F}^2_{L^2([-N \pi/T,N\pi/T])} \frac{1}{\E^{N^2 \pi^2 / (2 T^2 \varepsilon^2)} -1}.
\end{equation*}
Hence, combining this with the previous estimate and \eqref{lambda_norm_bd}, we obtain
\begin{align*}
\nm{g - h}_{L^2(\bbR)} + \tau \nm{\lambda}_{\ell^2} \lesssim & \sqrt{\int_{|\omega| > K \pi/T} | \hat{g}(\omega) |^2 \D \omega} 
\\
&+ \left ( \sqrt{\frac{T}{N}} \tau + \frac{\epsilon^{-1/2}}{\sqrt{\E^{N^2 \pi^2 / (2 T^2 \varepsilon^2)} -1}} \right ) \nm{F}_{L^2([-N \pi/T,N\pi/T])} . 
\end{align*}
Hence the result now follows by applying Lemma~\ref{lem:any_g_and_any_lambda} once more.
\end{proof}

The previous lemma involves an additional parameter $K$.  In practice, we make the specific choice $K = \min \{N,\varepsilon\}$.  This leads to the following:

\begin{lemma}
\label{lem:ENtau_main_bd_setK}
Let $f \in L^2(\Omega)$, and let $g \in L^2(\bbR)$ be any extension of $f$ to the real line.  If $\varepsilon$ satisfies \eqref{eps_min_simplify} then
\begin{align*}
E_{N,\tau}(f) \lesssim & \sqrt{\int_{|\omega| > \min\{N,\varepsilon\} \pi/T} | \hat{g}(\omega) |^2 \D \omega} 
\\
&+ \left( \frac{1}{\sqrt{\E^{N^2 \pi^2 / (2 T^2 \varepsilon^2)}-1}}+ \sqrt{\frac{T \varepsilon}{N}} \tau \right ) \E^{\pi^2/(4 T^2)} \nm{g}_{L^2(\bbR)}.
\end{align*}
\end{lemma}
\begin{proof}
We use Lemma~\ref{lem:ENtau_main_bd_general} with $K = \min\{N,\varepsilon\}$.  Observe that
\begin{align*}
\int_{|\omega| \leq K \pi/T} \frac{| \hat{g}(\omega) |^2 }{| \hat{\phi}(\omega/\varepsilon) |^2} \D \omega \lesssim \E^{K^2 \pi^2 / (2 T^2 \varepsilon^2)} \nm{g}^2_{L^2(\bbR)}.    
\end{align*}
The term $K^2 / \varepsilon^2 \leq 1$.  Hence 
\begin{equation*}
\int_{|\omega| \leq K \pi/T} \frac{| \hat{g}(\omega) |^2 }{| \hat{\phi}(\omega/\varepsilon) |^2} \D \omega \lesssim \E^{\pi^2 / (2 T^2)} \nm{g}^2_{L^2(\bbR)} .    
\end{equation*}
The result now follows from Lemma~\ref{lem:ENtau_main_bd_general}.
\end{proof}

\subsection{Limiting accuracy}\label{ss:limiting}
In this and the next section we analyze $E_{N,\tau}(f)$.  First, we consider the limiting behaviour as $N \rightarrow \infty$ for different choices of the parameter $\varepsilon$:

\begin{theorem}\label{thm:accuracylimit}
Let $f \in L^2(\Omega)$.  Then:
\begin{enumerate}
    \item If $\varepsilon \sim c N^{\alpha}$ for some $0 < \alpha < 1$ and $c > 0$ as $N \rightarrow \infty$ then
\begin{equation*}
E_{N,\tau}(f) \lesssim \tau \sqrt{T \epsilon / N} \E^{\pi^2/(4 T^2)} \nm{f}_{L^2(\Omega)} + o(1),\quad N \rightarrow \infty.
\end{equation*}
In particular, $E_{N,\tau}(f) \rightarrow 0$ as $N \rightarrow \infty$.
 \item If $\varepsilon \sim c N$ as $N \rightarrow \infty$ for some $c > 0$ then
 \begin{equation}\label{eq:linearscaling_limitingaccuracy}
E_{N,\tau}(f) \lesssim \left ( \frac{1}{\sqrt{\E^{\pi^2 / (2 c^2 T^2)} - 1}} + \sqrt{c T} \tau \right ) \E^{\pi^2 / (4 T^2)} \nm{f}_{L^2(\Omega)} + o(1).
\end{equation}
In particular, if $\tau < 1/2$ and $c$ satisfies
\begin{align}
\label{cT_choice_limit}
c T \leq \frac{\pi}{\sqrt{2 \log(1+\tau^{-2})}},  
\end{align}
then, as $N \rightarrow \infty$,
 \begin{equation*}
E_{N,\tau}(f) \lesssim \tau \E^{\pi^2 / (4 T^2)} \nm{f}_{L^2(\bbR)} + o(1).
\end{equation*}
\end{enumerate}
\end{theorem}

This result shows the following.  First, whenever $\varepsilon$ grows asymptotically more slowly than $N$ the error of the RBF approximation is guaranteed to approach zero as $N \rightarrow \infty$.  On the other hand, if $\varepsilon \sim c N$ then the error only decreases down to some finite value, with that value being dependent of the SVD regularization $\tau$, the extension domain size $T$ and the constant $c$.  In particular, if the product $c T$ is too large, the best achievable accuracy may be significantly larger than $\tau$.  On the other hand, by choosing $c T$ according to $\tau$ as in \eqref{cT_choice_limit} we ensure an asymptotic accuracy of order $\tau$.

\begin{proof}
We use Lemma~\ref{lem:ENtau_main_bd_setK} with $g \in L^2(\bbR)$ the extension of $f$ by zero.  Notice that $\nm{g}_{L^2(\bbR)} = \nm{f}_{L^2(\Omega)}$.  Since $\varepsilon$ satisfies \eqref{eps_min_simplify} we have, in all cases,
\begin{equation*}
    \int_{|\omega| > \min\{N,\varepsilon\} \pi/T} | \hat{g}(\omega) |^2 \D \omega \rightarrow 0,\qquad N \rightarrow \infty.
\end{equation*}
Moreover, if $\varepsilon \sim c N^{\alpha}$, then
\begin{equation*}
\frac{1}{\sqrt{\E^{N^2 \pi^2 / (2 \varepsilon^2 T^2)}-1}}+ \sqrt{\frac{\varepsilon T}{N}} \tau \sim \sqrt{\frac{\varepsilon T}{N}} \tau,\qquad N \rightarrow \infty,    
\end{equation*}
and if $\varepsilon \sim c N$ then
\begin{equation*}
\frac{1}{\sqrt{\E^{N^2 \pi^2 / (2 \varepsilon^2 T^2)}-1}}+ \sqrt{\frac{\varepsilon T}{N}} \tau \sim \frac{1}{\sqrt{\E^{\pi^2 / (2 c^2 T^2)}-1}}+ \sqrt{c T} \tau,\qquad N \rightarrow \infty.    
\end{equation*}
This gives the result.
\end{proof}

\subsection{Error decrease for smooth functions}\label{ss:smooth}

The previous result determines the limiting behaviour, but we are also interested in how fast the term $E_{N,\tau}(f)$ reaches that limit.  This, naturally, depends on the regularity of the function being approximated.  In this section, we study the behaviour for functions in the Sobolev spaces $H^k(\Omega)$. In order to obtain precise bounds, we now use exact regimes such as $\epsilon = c N^{\alpha}$, rather than the asymptotic scaling $\epsilon \sim c N^{\alpha}$ of the previous theorem.

\begin{theorem}\label{thm:convergence}
Let $f \in H^k(\Omega)$.  Then
\begin{enumerate}
\item Suppose that $\varepsilon = c N^{\alpha}$ for some $0 < \alpha < 1$ and $c > 0$.  Then
\begin{align*}
E_{N,\tau}(f) \lesssim \Big [ & (c N^{\alpha} \pi / T)^{-k}  
\\
& + \left (\E^{-\pi^2 N^{2(1-\alpha)} / (4 c^2 T^2)}  + \sqrt{c T}   \tau /N^{(1-\alpha)/2} \right ) \E^{\pi^2/(4 T^2)} \Big ] \nm{f}_{H^k(\Omega)},
\end{align*}
for all $N$ satisfying
\begin{equation*}
N^{1-\alpha} \geq \frac{\sqrt{2 \log(2)} T}{\pi} , \qquad c N^{\alpha} \geq \frac{1}{\sqrt{2}(T-B)} \sqrt{\log\left (\frac{\sqrt{2 \pi} B N}{T \tau^2} \right ) }  
\end{equation*}
\item Suppose that $\varepsilon = c N$ for some $c > 0$.  Then
\begin{align}
E_{N,\tau}(f) \lesssim \Big [ &(\min\{c,1\} N \pi/T)^{-k} \label{eq:linearscaling_error}
\\
& + \left ( \frac{1}{\sqrt{\E^{\pi^2 / (2 c^2 T^2)} - 1}} + \sqrt{c T} \tau \right ) \E^{\pi^2 / (4 T^2)} \Big ]   \nm{f}_{H^k(\Omega)}, \nonumber
\end{align}
for all $N$ satisfying
\begin{equation}
\label{eq:min_N_log_growth}
c N \geq  \frac{1}{\sqrt{2}(T-B)} \sqrt{\log\left (\frac{\sqrt{2 \pi} B N}{T \tau^2} \right ) }.
\end{equation}
In particular, if $\tau < 1/2$ and
\begin{equation}
\label{largest_c}
c \leq \min \left\{ 1 , \frac{\pi}{T \sqrt{2 \log(1+\tau^{-2})}} \right \},
\end{equation}
then
\begin{align*}
E_{N,\tau}(f) \lesssim \left ( (c N \pi / T)^{-k} +  \tau \E^{\pi^2 / (4 T^2)}  \right ) \nm{f}_{H^k(\Omega)}.    
\end{align*}
\end{enumerate}
\end{theorem}

The main conclusion is that whenever $\varepsilon$ grows more slowly than $N$ the algebraic convergence order is suboptimal.  Specifically, ${\mathcal O}(N^{-\alpha k})$, as opposed to the optimal ${\mathcal O}(N^{-k})$ rate.  Thus, this theorem suggests choosing $\varepsilon = c N$.   Seeking to make the algebraically-decreasing term in this case as small as possible, one is tempted to take $c$ as large as possible.  However, as was already noted in Theorem~\ref{thm:accuracylimit}, this limits the maximal achievable accuracy.  To attain the desired ${\mathcal O}(\tau)$ accuracy, we need to limit $c$ according to \eqref{largest_c}.  Hence, the overall conclusion is to choose $\varepsilon = c N$ with
\begin{equation}\label{eq:optimal_c}
c = c^* = \frac{\pi}{T \sqrt{2 \log(1+\tau^{-2})}}.
\end{equation}
This also suggests choosing $T > B$ as small as possible.  However, because of the condition \eqref{eq:min_N_log_growth}, choosing $T$ too small may result in a longer pre-asymptotic regime of the error, before the error rates of Theorem~\ref{thm:convergence} kick in.

\begin{proof}[Proof of Theorem~\ref{thm:convergence}]
By the Sobolev extension theorem, there exists an extension $g \in H^k(\bbR)$ of $f$ satisfying $\nm{g}_{H^k(\bbR)} \lesssim \nm{f}_{H^k(\Omega)}$.  Recall that, by definition,
\begin{equation*}
\nm{g}^2_{H^k(\bbR)} = \int^{\infty}_{-\infty} (1+|\omega|)^{2k} | \hat{g}(\omega) |^2 \D \omega.    
\end{equation*}
Hence
\begin{align*}
\sqrt{\int_{|\omega| > \min\{N,\varepsilon\} \pi/T} | \hat{g}(\omega) |^2 \D \omega} \lesssim (\min\{N,\varepsilon\} \pi/T)^{-k} \nm{f}_{H^k(\Omega)}.
\end{align*}
We use Lemma~\ref{lem:ENtau_main_bd_setK}.  The case for $\epsilon = c N$ follows immediately.  For the case $\epsilon = c N^{\alpha}$ we merely notice that
\begin{align*}
\E^{N^2 \pi^2 / (2 \varepsilon^2 T^2)}-1  \geq \frac12 \E^{N^2 \pi^2 / (2 \varepsilon^2 T^2)} = \frac12 \E^{N^{(2-\alpha)} \pi^2 / (2 c^2 T^2)}, 
\end{align*}
due to the assumption $N$.  This completes the proof.
\end{proof}

\section{Numerical experiments}
\label{sec:experiments}

We illustrate the methodology of this paper with a sequence of examples. For the case of univariate approximation with a Gaussian radial function, these results confirm the results of the theory.

We vary the scaling of the shape parameter $\varepsilon$ with $N$, the shape of the domain $\Omega$ as well as its dimension and the radial function $\phi$. In all examples, the centers $\{ \xi_n\}$ are located in the bounding box $[-T,T]^d$, whereas the sampling points $\{ x_m\}$ are restricted to lie inside $\Omega$. We solve the rectangular least squares problem using a truncated SVD decomposition with a regularization threshold $\tau$. All singular values smaller than $\tau \sigma_1$, with $\sigma_1$ the largest one, are discarded.

\subsection{Centers outside the computational domain and the choice of $\boldmath{T}$}

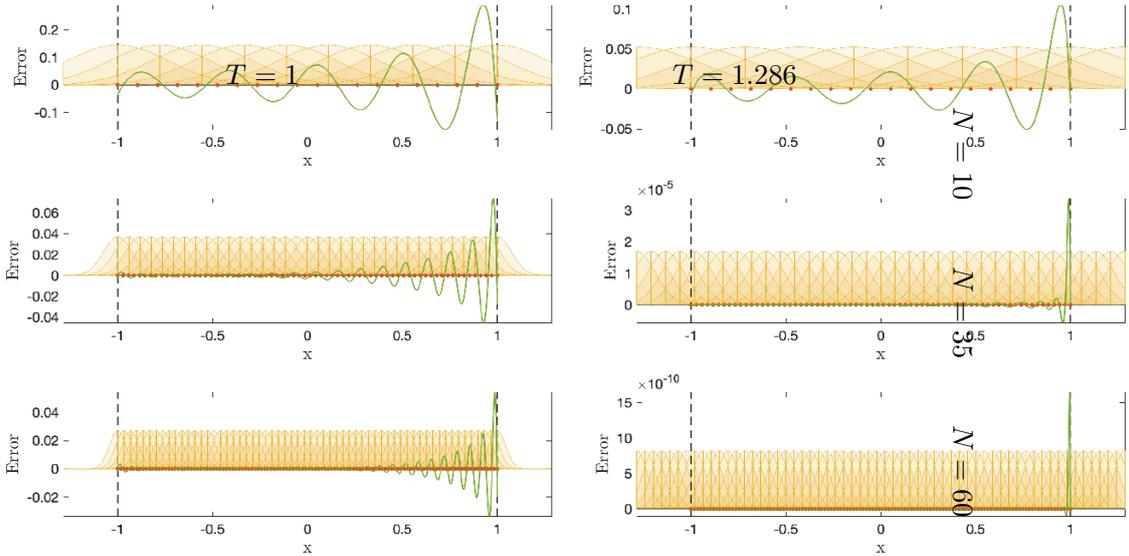
\begin{figure}
    \centering
    \begin{tikzpicture}[scale=0.98, every node/.style={transform shape}]
    \node[anchor=south west,inner sep=0] at (0,0) {\includegraphics[trim=1cm 1cm 2cm 2cm, clip,width=\textwidth]{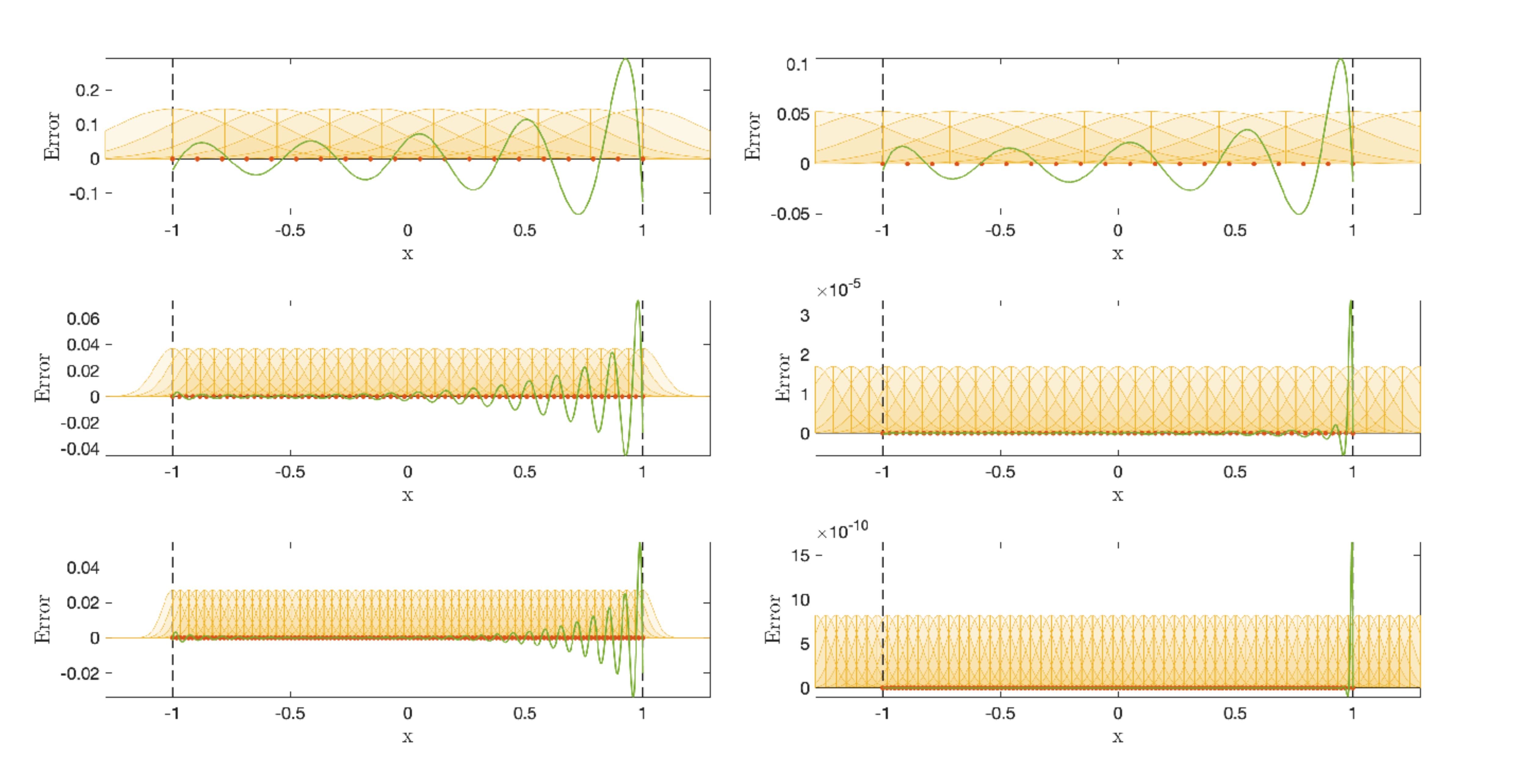}};
    \node[text width=2cm] at (4,6.6) {$T=1$};
    \node[text width=2cm] at (10,6.6) {$T=1.286$};
    \node[text width=2cm, rotate=-90] at (12.9,5.15) {$N=10$};
    \node[text width=2cm, rotate=-90] at (12.9,3) {$N=35$};
    \node[text width=2cm, rotate=-90] at (12.9,0.85) {$N=60$};
\end{tikzpicture}
    \caption{Illustration of the error in approximating $f(x)=\frac{1}{x-1.2}$ when (left) all centers lie within the domain $[-1,1]$ and (right) some centers lie outside it. The green curves show the difference between the true and the approximated functions, and the yellow curves show the radial function translates centered at each one of the $N$ centers. There are twice as many sample points as centers. The error in the left column never decays to $0$ while the convergence is fast in the right column. Without centers outside the domain, the RBF translates are not sufficiently dense near the endpoints.}
    \label{fig:fig_5_1}
\end{figure}

First, we illustrate the importance of adding centers outside the computational domain in the linear scaling regime $\varepsilon = cN$. We consider the interval $\Omega = [-1,1]$. Intuitively, due to the truncation at the endpoints, the radial basis functions lack sufficient density near $\pm 1$ when centers are confined to $[-1,1]$. This is clearly visible in Fig.~\ref{fig:fig_5_1}, in which the translations of the Gaussian radial basis function are plotted with and without centers outside $[-1,1]$. Loosely speaking, the approximation space is enlarged by adding centers outside the domain, until the leftmost and rightmost RBF are numerically small on the computational domain itself.

More precisely, our analysis has shown that convergence sets in once $N$ is larger than a minimal value reflected in the bound \eqref{eq:min_N_log_growth}. This bound merely ensures the above-mentioned condition. Indeed, note that the rightmost RBF in \eqref{rbf_sum_equispaced} is $\phi(\varepsilon(x - T))$. Letting $\varepsilon=cN$ equal the right-hand side  of \eqref{eq:min_N_log_growth} and evaluating this RBF at the point $x=1$ yields, noting that $B=1$ in this case and recalling the scaling by $\sqrt{\varepsilon}$,
\[
 \sqrt{\varepsilon} \phi(\varepsilon(1 -T)) = \sqrt{cN} \phi\left(\frac{1}{\sqrt{2}}\sqrt{\log\left(\frac{\sqrt{2\pi}N}{T\tau^2}\right)}\right) = \sqrt{cN}\E^{-\frac12\log\left(\frac{\sqrt{2\pi}N}{T\tau^2}\right)} = \sqrt{\frac{cT}{\sqrt{2\pi}}} \tau.
\]
That is the maximal value of the $N$-th radial basis function on $[-1,1]$: it is on the order of $\tau$.

\subsection{Scaling regimes of the shape parameter}

We consider three regimes for the shape parameter $\varepsilon(N)$: (i) $\varepsilon$ grows like $cN$, (ii) it grows like $c \sqrt{N}$ or (iii) it is constant.

\subsubsection{Linear scaling: $\varepsilon = cN$}

\begin{figure}
    \centering
    \begin{minipage}{0.5\textwidth}
        \centering
        \includegraphics[width=1\textwidth]{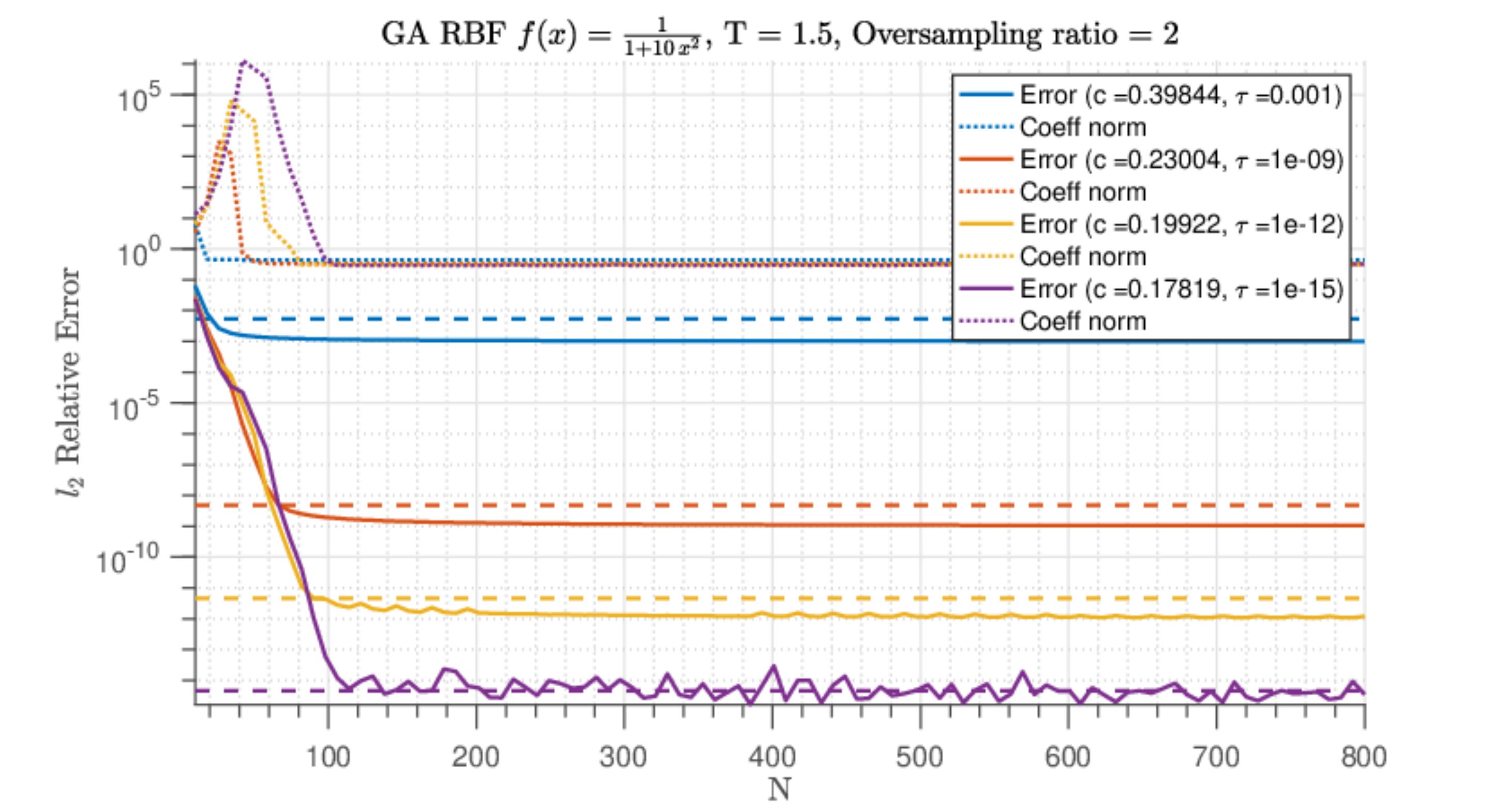}
    \end{minipage}\hfill
    \begin{minipage}{0.5\textwidth}
        \centering
        \includegraphics[width=1\textwidth]{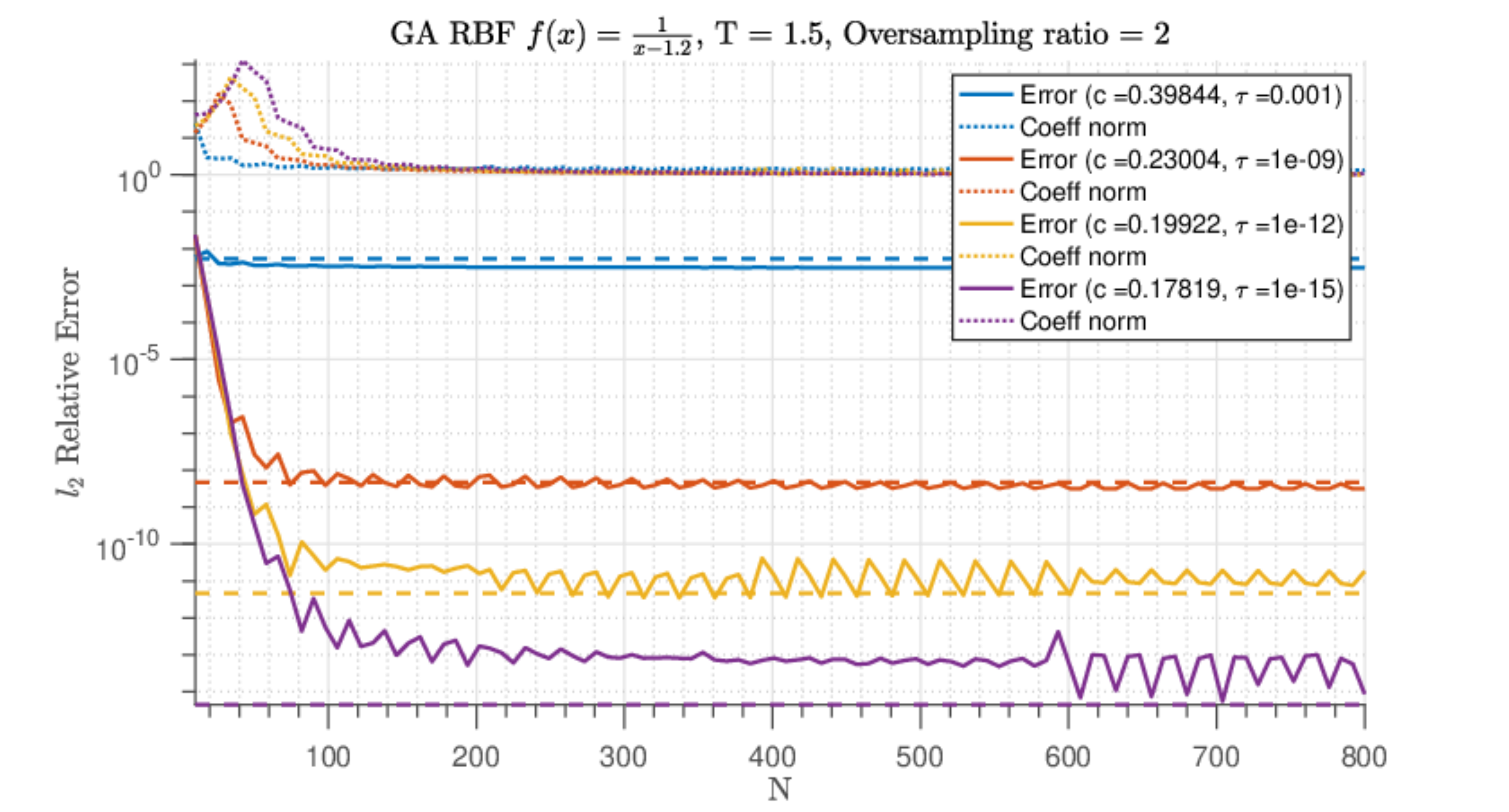}
    \end{minipage}
        \begin{minipage}{0.5\textwidth}
        \centering
        \includegraphics[width=1\textwidth]{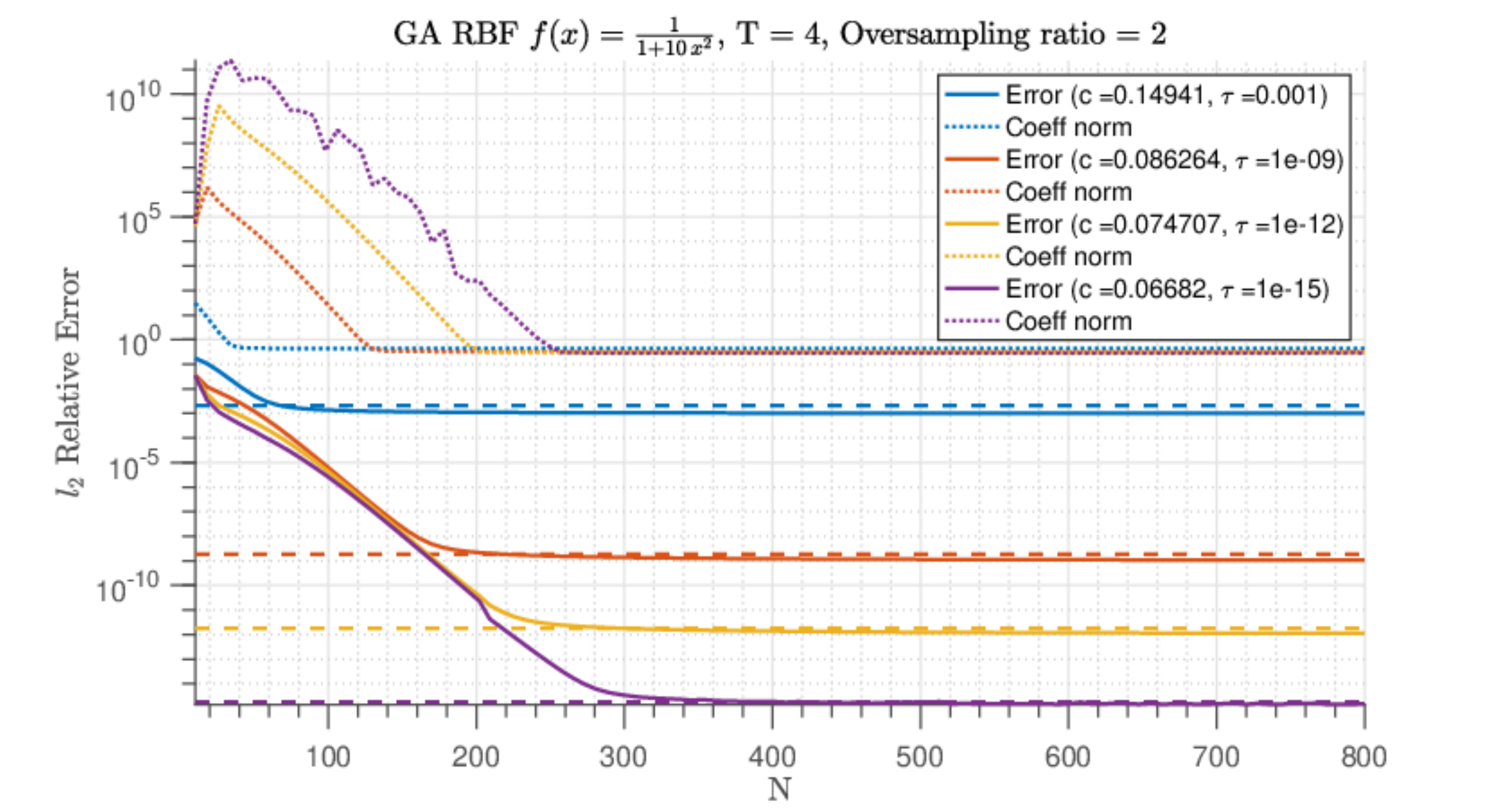} 
    \end{minipage}\hfill
    \begin{minipage}{0.5\textwidth}
        \centering
        \includegraphics[width=1\textwidth]{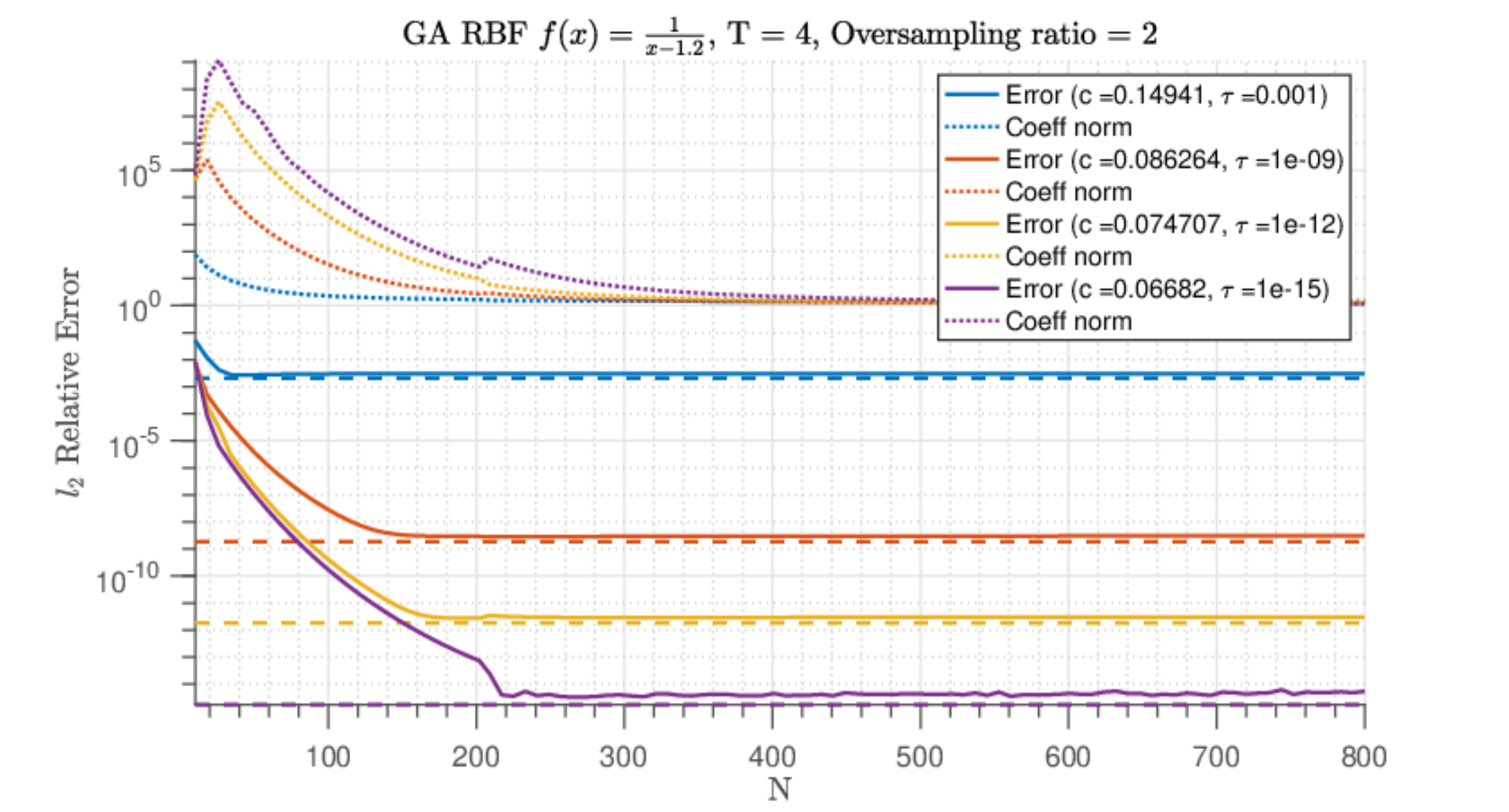}
    \end{minipage}
    \caption{Error and magnitude of the coefficients in the linear scaling regime $\varepsilon=c N$ for the functions $f(x)=\frac{1}{1+10\,x^2}$ (left column) and $f(x) = 1/(x-1.2)$ (right column). Results are shown for varying values of the regularization threshold $\tau$. Different values of $T$ are used in the top row ($T=1.5$) and bottom row ($T=4$). The proportionality constant $c$ is related to $T$ and $\tau$ via \eqref{eq:optimal_c}. Shown with horizontal dashed lines is the estimate of the limiting accuracy~\eqref{eq:linearscaling_limitingaccuracy}.}\label{fig:linearscaling_optimalc}
\end{figure}

Results for the linear scaling regime $\varepsilon = cN$ are shown in Fig.~\ref{fig:linearscaling_optimalc} for two functions approximated on $[-1,1]$: $f(x)=1/(1+10x^2)$, the Runge function with poles near the real line in the middle of the interval (left column), and $f(x) = 1/(x-1.2)$ with a pole near the right endpoint (right column). We have used an oversampling factor $\gamma=2$. Results are shown for different values of the threshold $\tau$, indicated on the figure, and with $c$ chosen depending on $T$ and $\tau$ according to our estimate of the optimal value \eqref{eq:optimal_c}.

In all cases, there is an initial regime in which the coefficient norm grows. Smaller values of the threshold $\tau$ result in larger coefficient norms. Note that decreasing $\tau$ leads to decreasing $c$ and decreasing $\varepsilon$ and, hence, to flatter radial basis functions. However, as the approximation converges to its best accuracy, the coefficient norm also settles down to a moderate value. This initial regime of coefficient growth can potentially be avoided by using the method introduced in~\cite{adcock2021bounded}.

Seemingly geometric convergence is observed in all cases, because the functions involved are analytic on $[-1,1]$. In the linear scaling regime the maximal achievable accuracy is limited and depends on $c$, $T$ and $\tau$. The limiting accuracy predicted by Theorem~\ref{thm:accuracylimit}, equation~\eqref{eq:linearscaling_limitingaccuracy}, is
\begin{equation*}
\left(\frac{1}{\sqrt{e^{\pi^2/(2 c^2 T^2)} - 1}} + \sqrt{c T} \tau \right) e^{\pi^2/(4 T^2)}.
\end{equation*}
This value is included in the figures with horizontal dashed lines. There is very good agreement with this limit in all cases.

Comparing the top and bottom row illustrates that convergence rate in this example is better for smaller $T$, in agreement with Theorem~\ref{thm:convergence}. As long as \eqref{cT_choice_limit} is satisfied, theory also predicts that smaller $T$ is better. We omit further experiments involving varying $T$ but note that a thorough investigation of the influence of the extension parameter $T$ was given for Fourier extension approximations in~\cite{adcock2014resolution}.

\begin{figure}
    \centering
    \begin{minipage}{0.5\textwidth}
        \centering
        \includegraphics[width=1\textwidth]{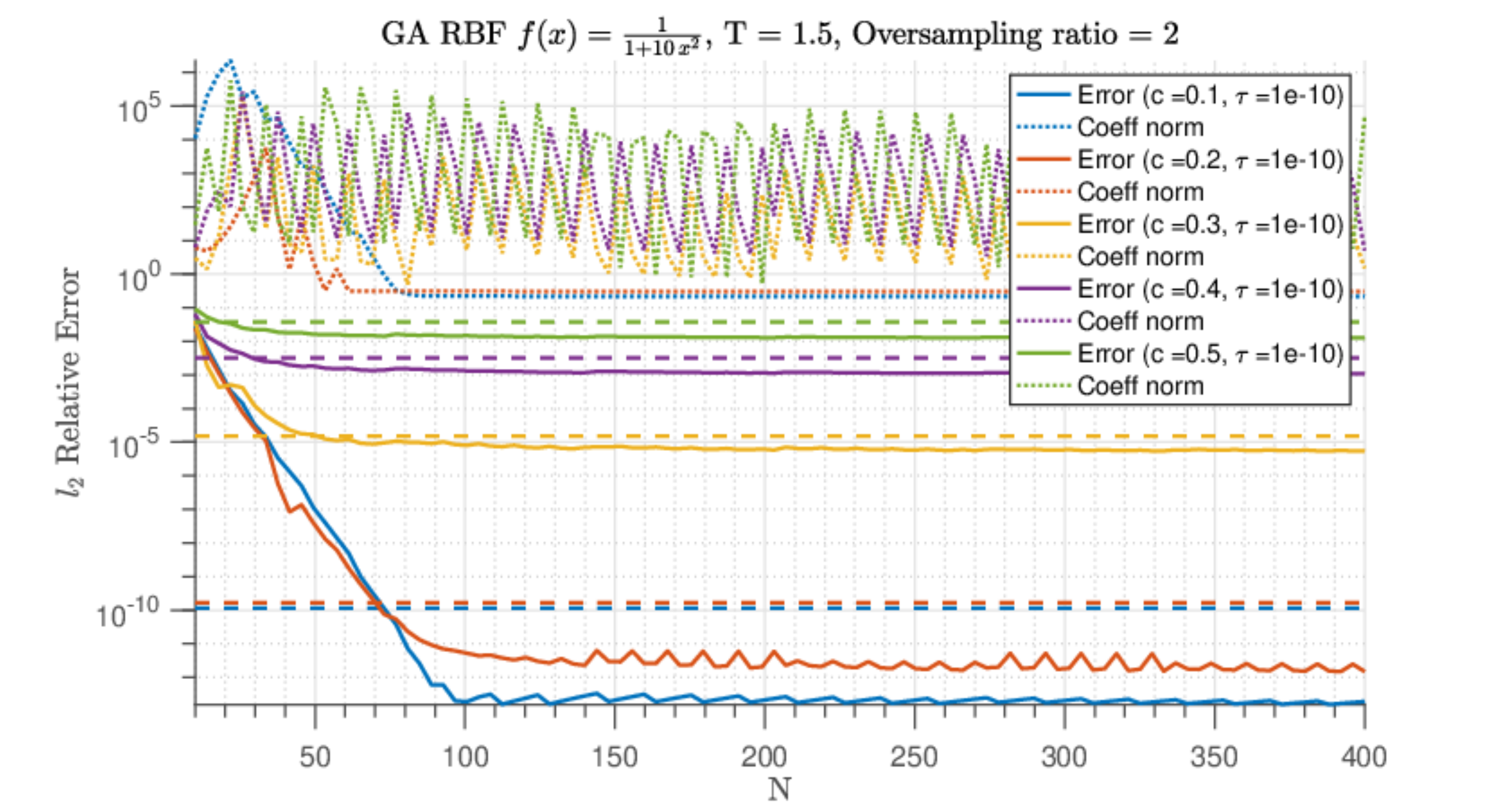}
    \end{minipage}\hfill
    \begin{minipage}{0.5\textwidth}
        \centering
        \includegraphics[width=1\textwidth]{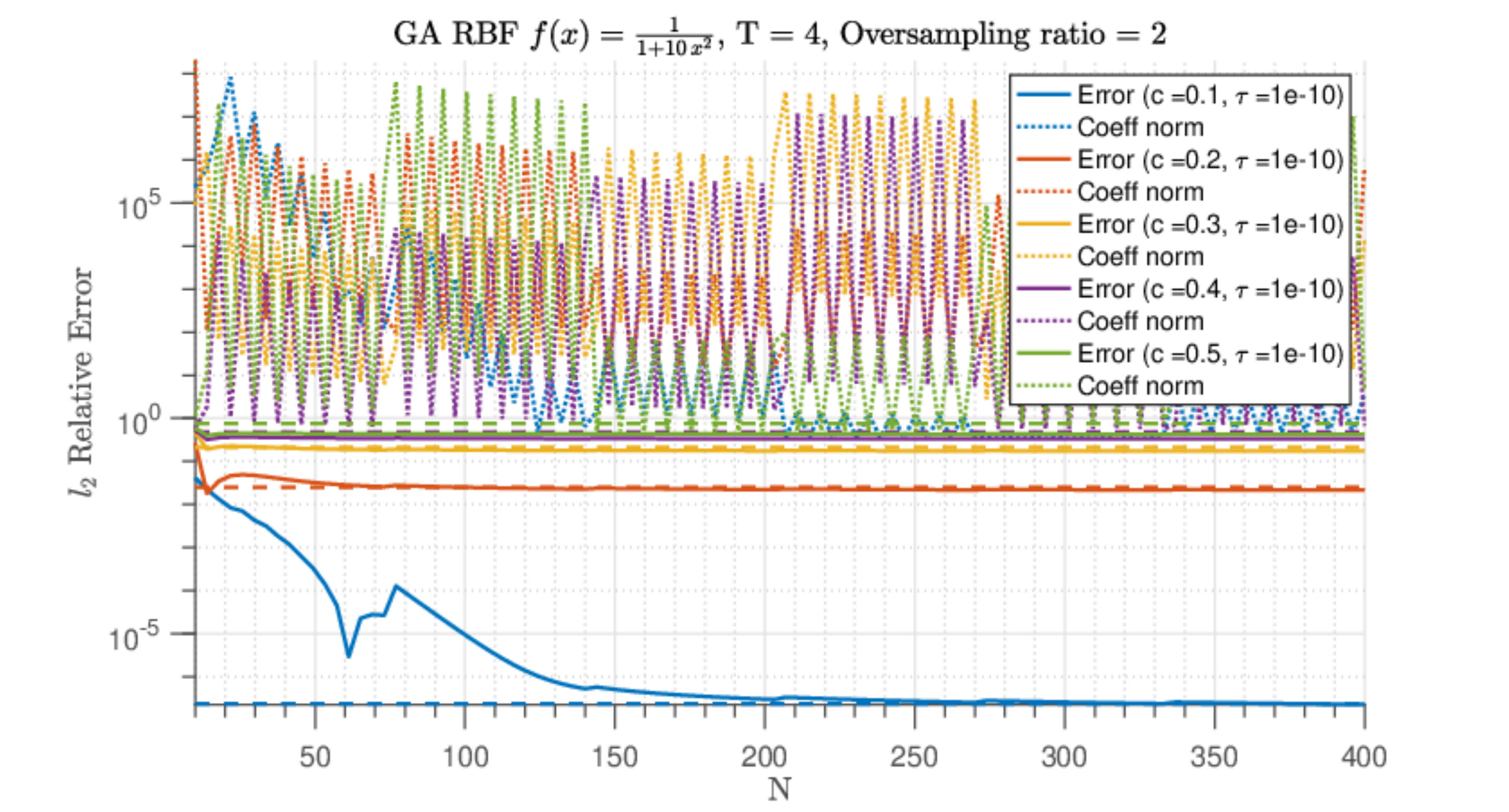}
    \end{minipage}
    \caption{This figure is similar to Fig.~\ref{fig:linearscaling_optimalc}, but here we fix the threshold $\tau= 1e-10$ and vary the proportionality constant $c$ of the linear scaling regime $\varepsilon = cN$. Suboptimal values of $c$ lead to saturation of the error at much higher levels than $\tau$.}\label{fig:linearscaling_fixedc}
\end{figure}

We repeat the experiment in Fig.~\ref{fig:linearscaling_fixedc} for a fixed value of $\tau= 1e-10$ and varying $c$. There is still a good agreement between the maximal achievable accuracy and the predicted accuracy limit. However, in this regime we clearly observe saturation errors: the accuracy limit can be much higher than the threshold value $\tau$. For values of $c$ where this is the case (e.g. $c=0.5$), the coefficient norm also exhibits erratic behaviour as a function of $N$.

\subsubsection{Sublinear scaling: $\varepsilon = c \sqrt{N}$}\label{sss:sublinear}

 \begin{figure}
    \centering
    \begin{minipage}{0.5\textwidth}
        \centering
        \includegraphics[width=1\textwidth]{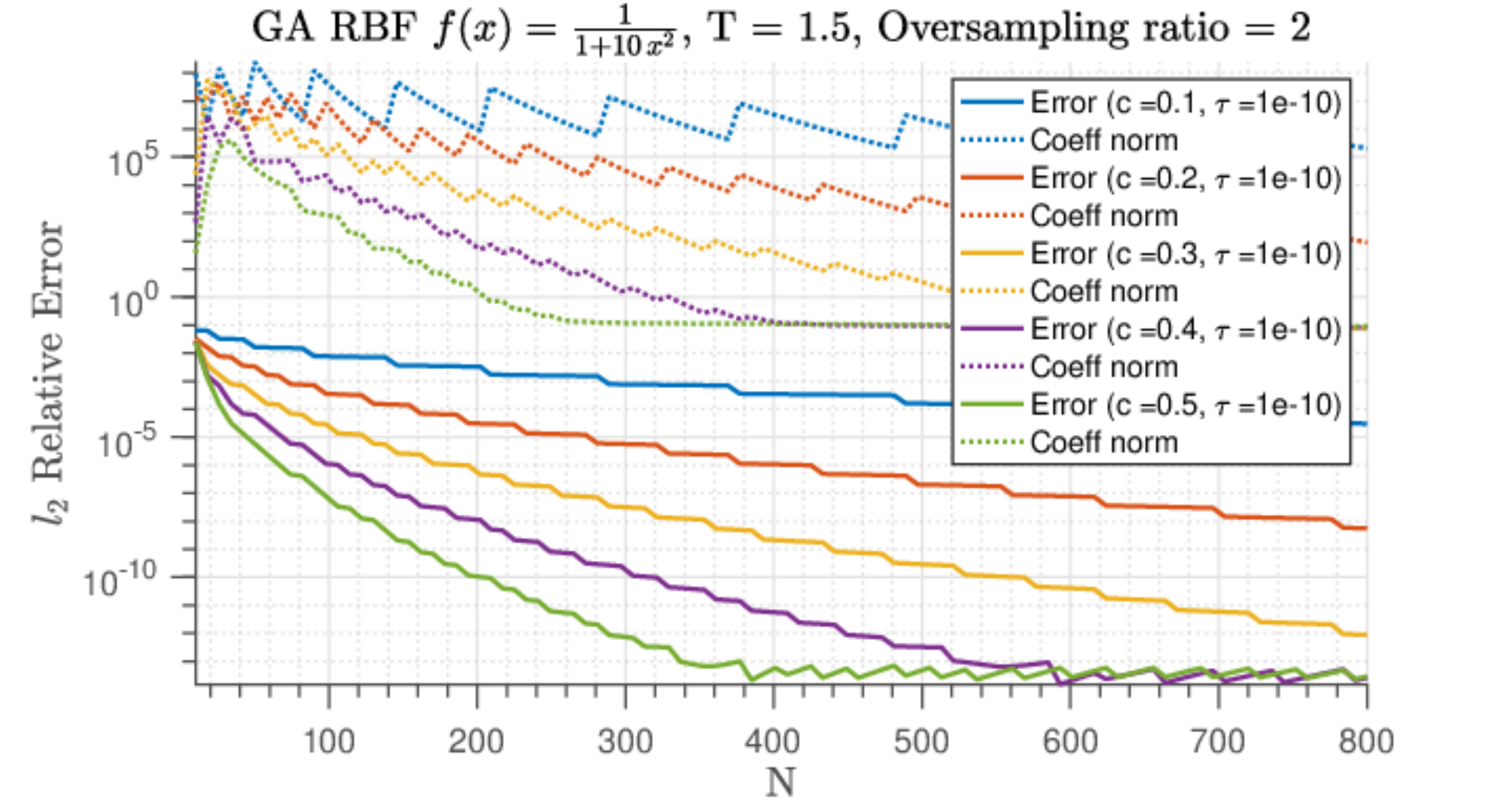}
    \end{minipage}\hfill
    \begin{minipage}{0.5\textwidth}
        \centering
        \includegraphics[width=1\textwidth]{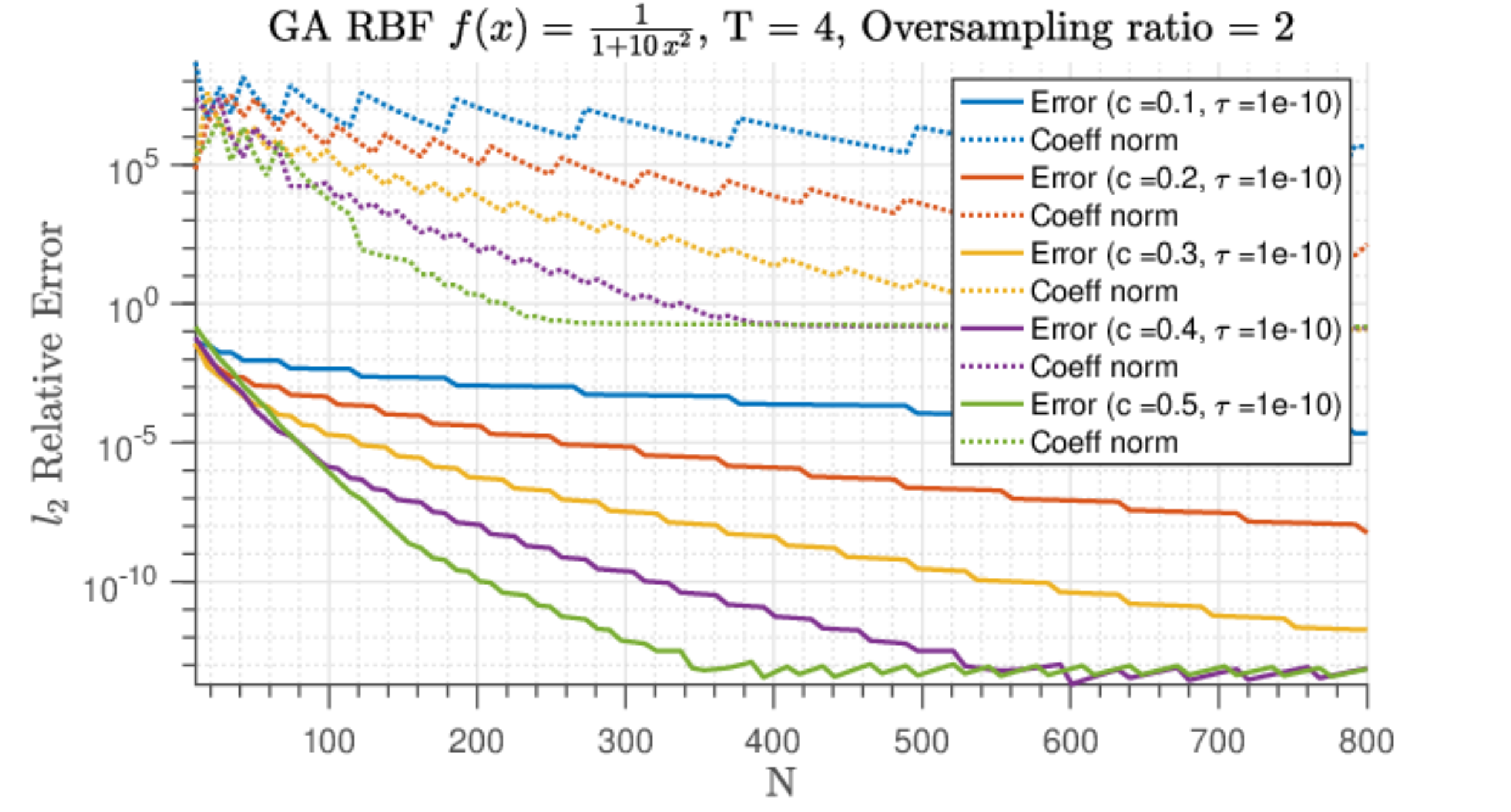}
    \end{minipage}
    \caption{Error and magnitude of the coefficients in the square-root scaling regime $\varepsilon = c \sqrt{N}$ for function $f(x)=\frac{1}{1+10\,x^2}$. Different values of $T$ are used in the left and right panel. Each plot includes varying values of $c$, independently of the fixed threshold $\tau = 1e-10$.}\label{fig:sublinearscaling}
\end{figure}

Results are shown in Fig.~\ref{fig:sublinearscaling} for the scaling regime $\varepsilon = c \sqrt{N}$. It follows from Theorem~\ref{thm:accuracylimit} that there is no saturation error in this regime: the expected limiting accuracy for large $N$ is $\mathcal{O}(\tau)$, independently of $c$. However, it follows from Theorem~\ref{thm:convergence} that larger values of $c$ asymptotically lead to more rapid convergence. Both observations are confirmed in Fig.~\ref{fig:sublinearscaling}. The theory makes no claim about the pre-asymptotic regime, and we note that in the right panel of the figure a larger value of $c$ can initially be worse than a smaller one. Once the asymptotic regime of convergence is reached, larger $c$ leads to faster decay of the error as well as to a smaller coefficient norm.

For comparison, we include a result on the approximation of a less smooth function in Fig.~\ref{fig:fig_5_2_2}. To that end we consider the function $f(x) = |x|^5$. The convergence rates are algebraic in this case, but the results are otherwise similar: the linear regime offers faster convergence, up to a saturation limit. The sublinear regime leads to convergence, but does so more slowly.

\begin{figure}
    \centering
    \includegraphics[width=0.7\textwidth]{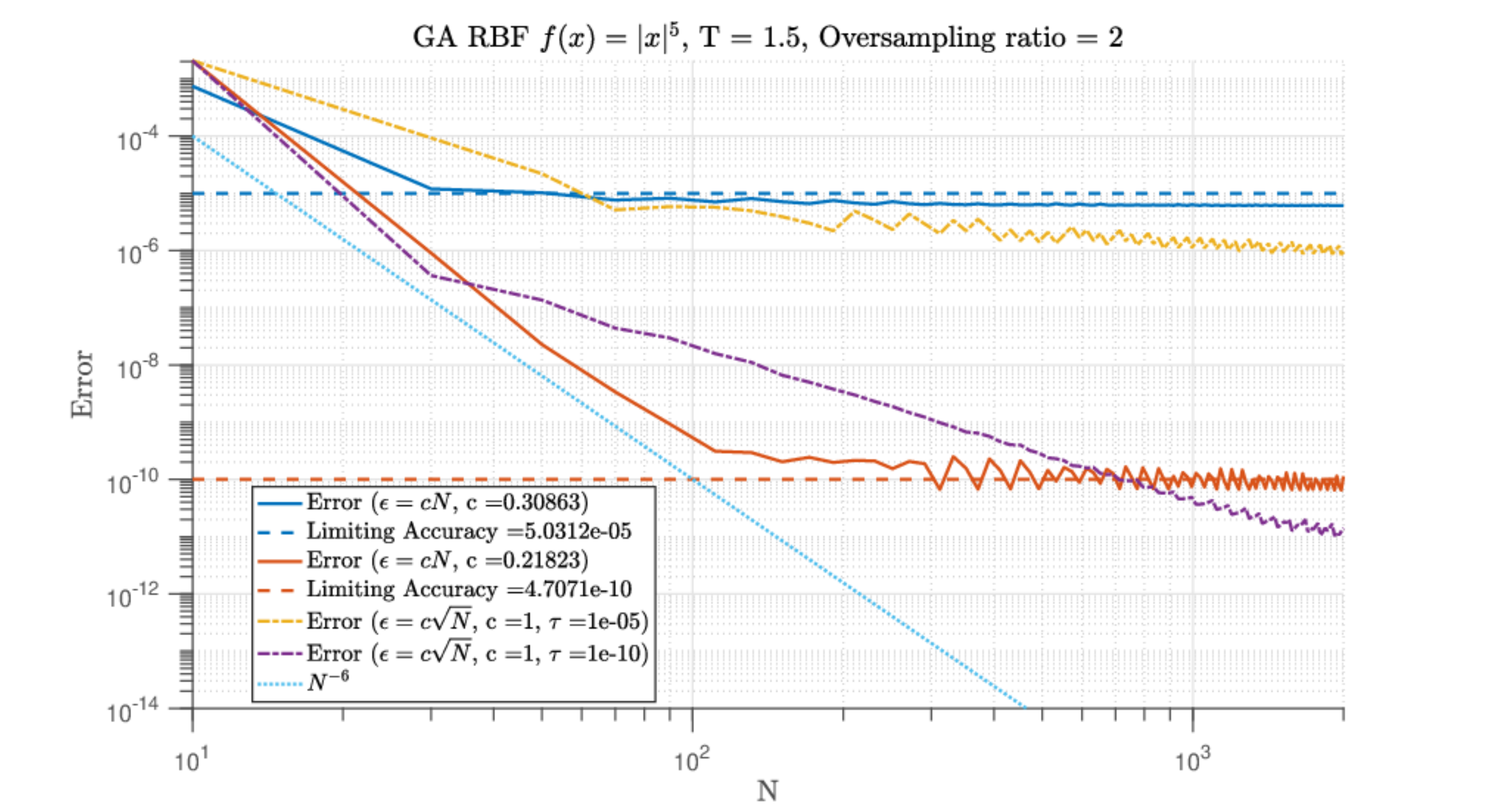}
    \caption{Log-log plot of the error decay using linear and sublinear $\varepsilon$ values when approximating function $f(x)=|x|^5$, for $\tau = 10^{-5}$ and for $\tau = 10^{-10}$. The decay of the error is algebraic in this case, since the target function is of finite regularity. We display the $\log(error)$ versus $\log(N)$ along with some powers of $\frac{1}{N}$ to help gauge the decay. As expected, the linear scaling of $\varepsilon$ leads to a fast convergence, but does not decay past a best error value. On the other hand, a sublinear scaling of the shape parameter leads to a slower decay of the error, and never reaches a minimum error threshold.}
    \label{fig:fig_5_2_2}
\end{figure}

\subsubsection{The flat limit: $\varepsilon = c$}
In the flat limit in $1-$D, one can easily show that the radial function approximant converges to a polynomial over $[-T,T]$, once a stable method has been applied. However, none of these strategies, mentioned in \S \ref{ssec:rbf_approx}, are part of our setting and we will therefore not consider this case.

\subsection{Families of radial basis functions}

Our analysis and the previous experiments are based on the Gaussian RBF. From the point of view of the analysis developed in this paper, and the specific methodologies of proof that we employed, there is a large qualitative difference between compactly supported and non-compactly supported RBF's. However, the main observation that inspired our analysis, namely the interplay of numerical stability with coefficient norm and approximation error, remains the same: the size of the expansion coefficients limits the potentially achievable accuracy in numerical computations.

Here, we repeat the approximation of the Runge function $f(x) = \frac{1}{1+10x^2}$ using four different families of radial basis functions: the multiquadric (MQ), inverse quadratic (IQ), inverse multiquadric (IMQ) and the Gaussian (GA) (recall their definitions in Table~\ref{RadialFunctions}). We always use equispaced centers on $[-T,T]$ and oversampling with an oversampling ratio of $\gamma$.

Results are shown in Fig.~\ref{fig:radialfunctions0} for $T=4$ and $\gamma=2$. For reference, in each figure we have included the limiting accuracy of Gaussian radial basis function as predicted by Theorem \ref{thm:accuracylimit} in dashed lines. We note that in all cases the coefficient norm initially increases, before decreasing down to a limit. In this experiment, it does seem to be the case that the Gaussian RBF offers highest convergence rates, the highest accuracy and the smallest coefficient norm. The loss of accuracy compared to the threshold $\tau$ in the other RBF families is due to a larger coefficient norm. These results have motivated our selection of the Gaussian RBF for further analysis.

Qualitatively similar results are obtained for other choices of $T$ and $\gamma$, except that (as with the Gaussian RBF) accuracy may be lost in the absence of sufficient oversampling. We study the effect of oversampling separately further on.

\begin{figure}[ht]
  \begin{minipage}[b]{0.5\linewidth}
    \includegraphics[width=.98\linewidth]{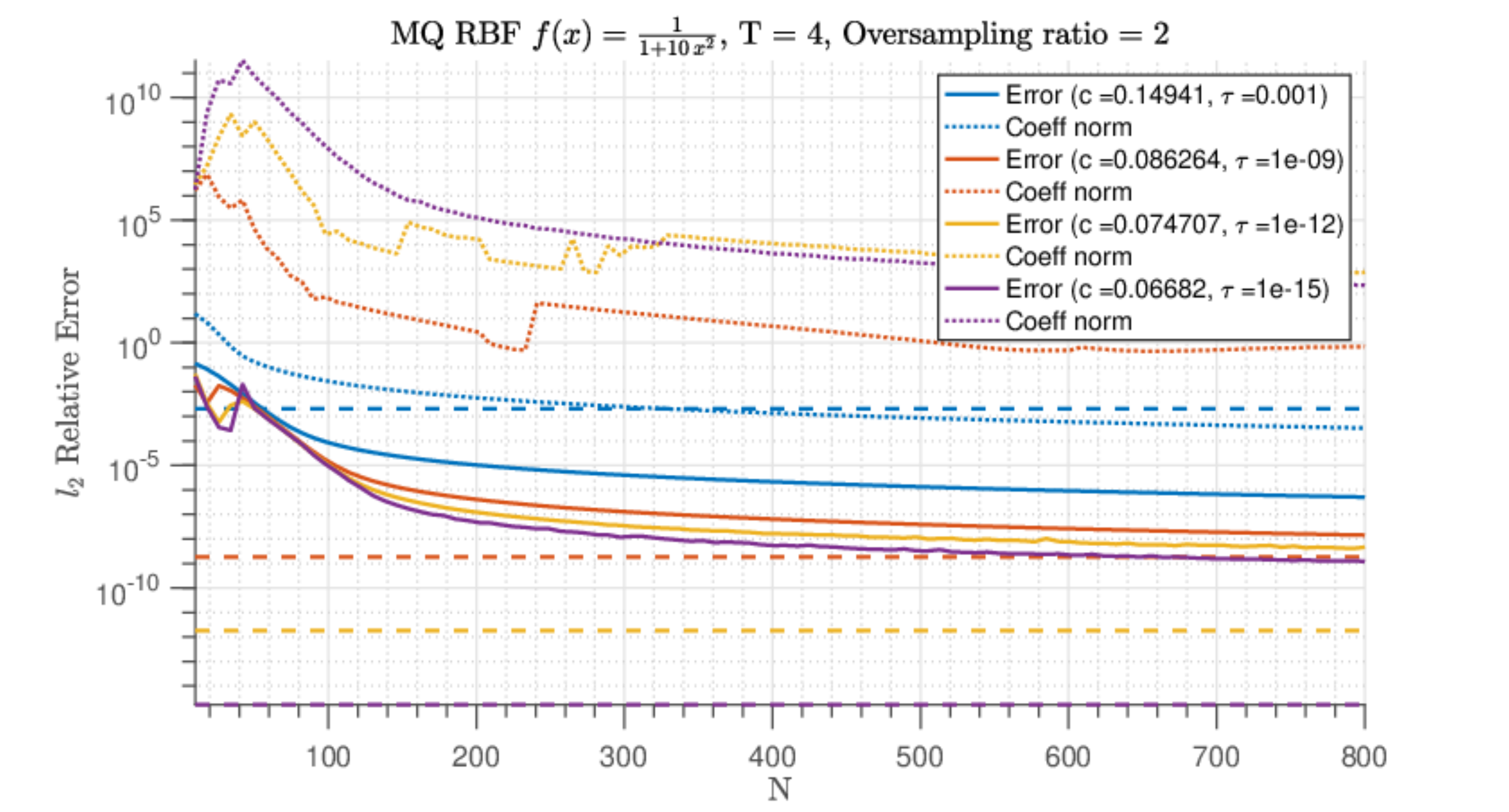} 
  \end{minipage} 
  \begin{minipage}[b]{0.5\linewidth}
    \includegraphics[width=.98\linewidth]{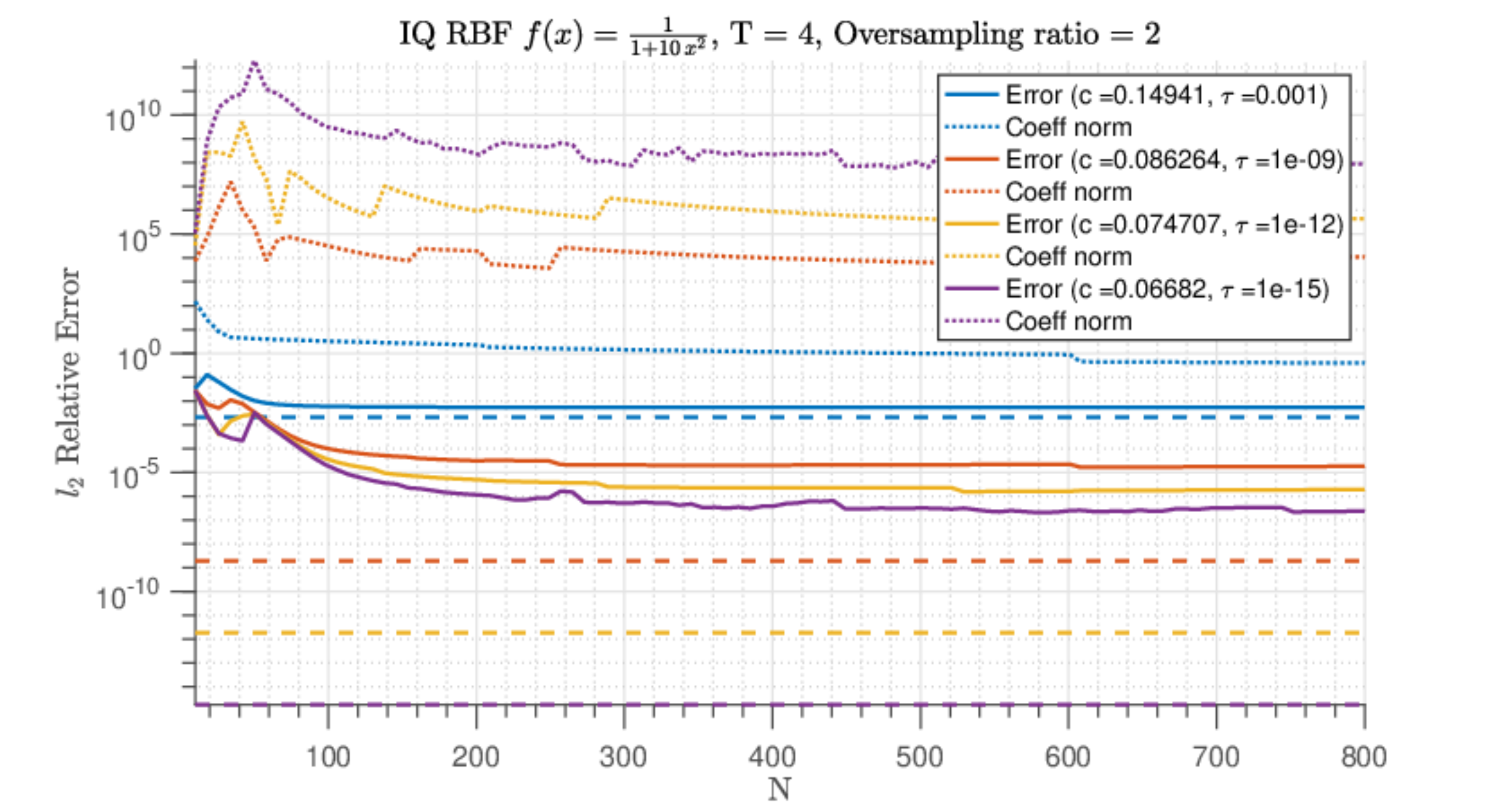} 
  \end{minipage} 
  \begin{minipage}[b]{0.5\linewidth}
    \includegraphics[width=.98\linewidth]{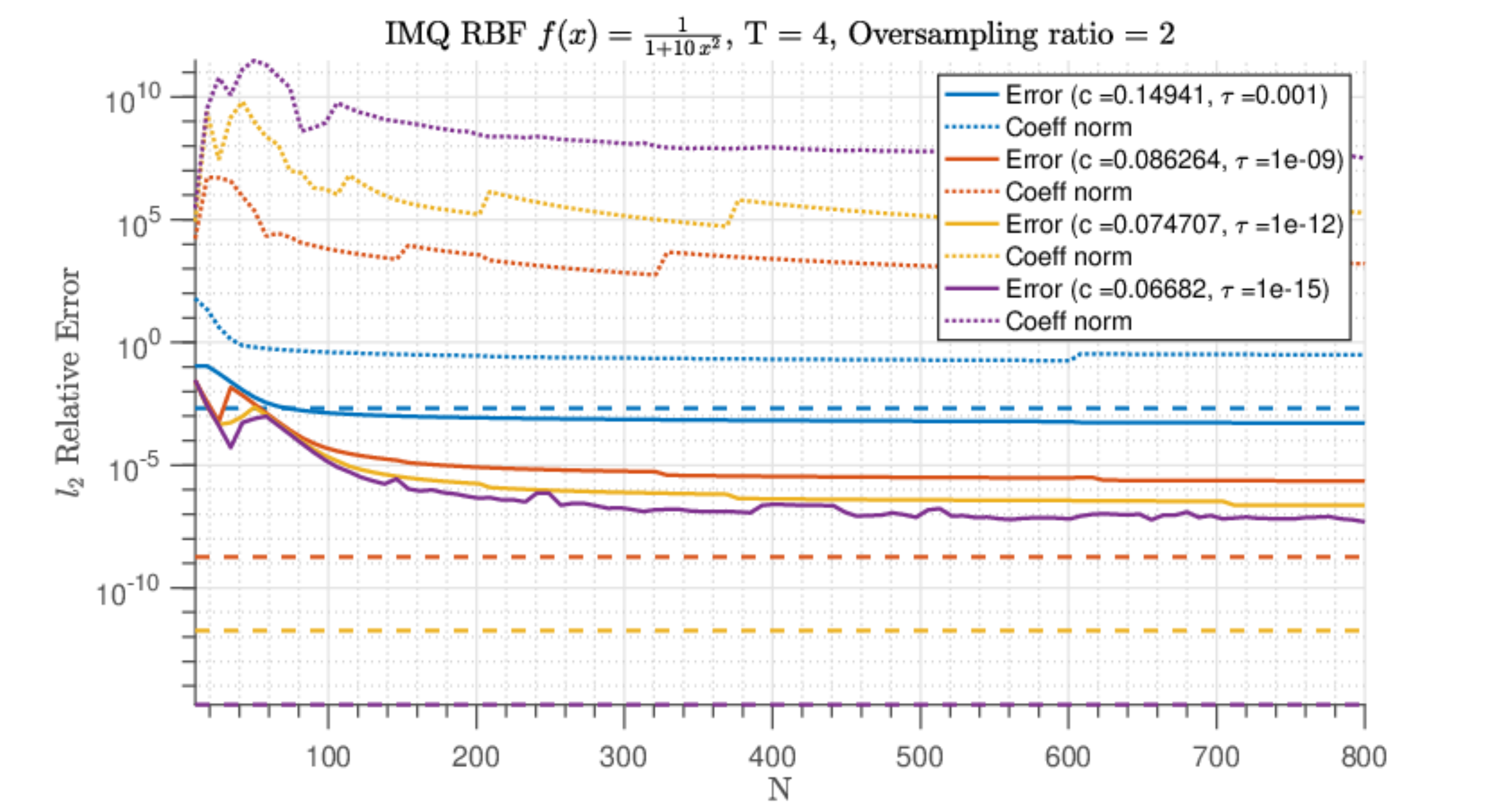}
  \end{minipage}
  \hfill
  \begin{minipage}[b]{0.5\linewidth}
    \includegraphics[width=.98\linewidth]{cError_CoeffNorm_func1_or_2_T_4_RBF_GA.pdf}
  \end{minipage}
  \caption{Approximation of the Runge function $f(x) = \frac{1}{1+10x^2}$ for different radial basis functions: multiquadric (MQ), inverse quadratic (IQ), inverse multiquadric (IMQ) and Gaussian (GA). The approximation is based on equispaced centers in $[-T,T]$ with $T=4$ and oversampling by a factor of $\gamma=2$. The horizontal dashed lines show the expected limiting accuracy \eqref{eq:linearscaling_limitingaccuracy} for GA.}\label{fig:radialfunctions0}
\end{figure}

\subsection{The sampling rate and a rule of thumb}

In our theoretical analysis, we have investigated the minimization problem~\eqref{eq:ENeps}. We have argued that, provided there is sufficient oversampling, the numerical result corresponds to this optimum. A more general minimization problem which does take sampling into account, via the constants $c_{M,N}^\tau$ and $d_{M,N}^\tau$, was \eqref{eq:fna2_error}.

In practice, any difference between \eqref{eq:fna2_error} and \eqref{eq:ENeps} can be easily detected. As a rule of thumb, if \eqref{eq:ENeps} is minimized, then one expects the two terms in this minimization problem to be balanced and, thus, one expects the ratio of the approximation error over the coefficient norm to approximately equal $\tau$:
\begin{equation}\label{eq:ruleofthumb}
\frac{\nm{f - \mathcal{T}_N z }_{L^2(\Omega)}}{\nm{z}_{\ell^2}} \approx \tau.
\end{equation}
Conversely, if this ratio does not approximate $\tau$, then \eqref{eq:ENeps} was not optimally solved and one reason may be that there was not enough oversampling. This means that the error is governed by the original expression \eqref{eq:fna2_error} in which $c_{M,N}^\tau$ and $d_{M,N}^\tau$ are larger than expected. The solution in this case is to increase $M$ with respect to $N$. Recall that we assume that the centers ultimately (i.e., for large $M$) fill the domain $\Omega$ in a quasi-uniform manner, to ensure that $c_{M,N}^\tau, d_{M,N}^\tau \lesssim 1$.

\begin{figure}
    \centering
    \begin{minipage}{0.5\textwidth}
        \centering
        \includegraphics[width=1\textwidth]{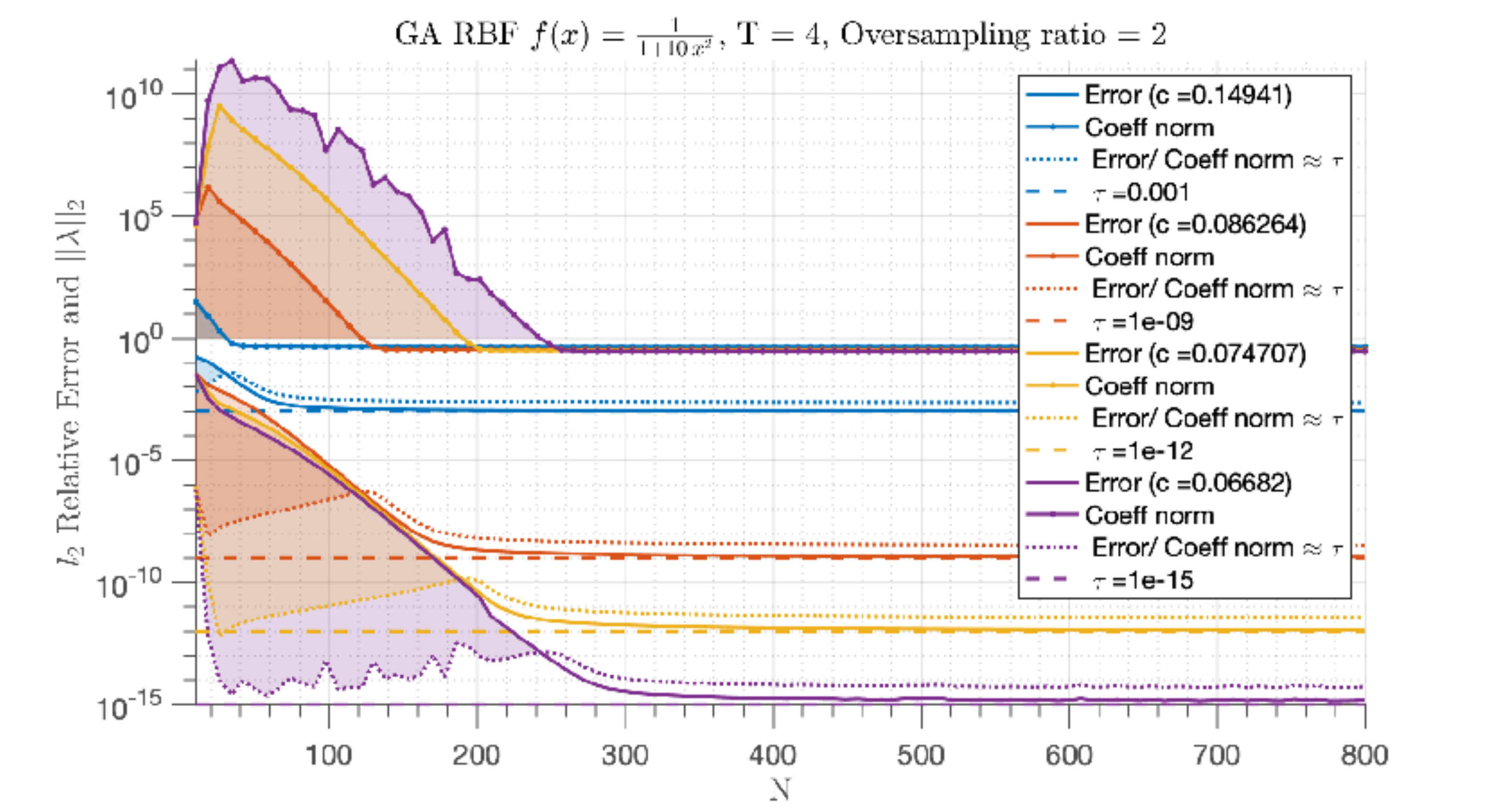}
    \end{minipage}\hfill
    \begin{minipage}{0.5\textwidth}
        \centering
        \includegraphics[width=1\textwidth]{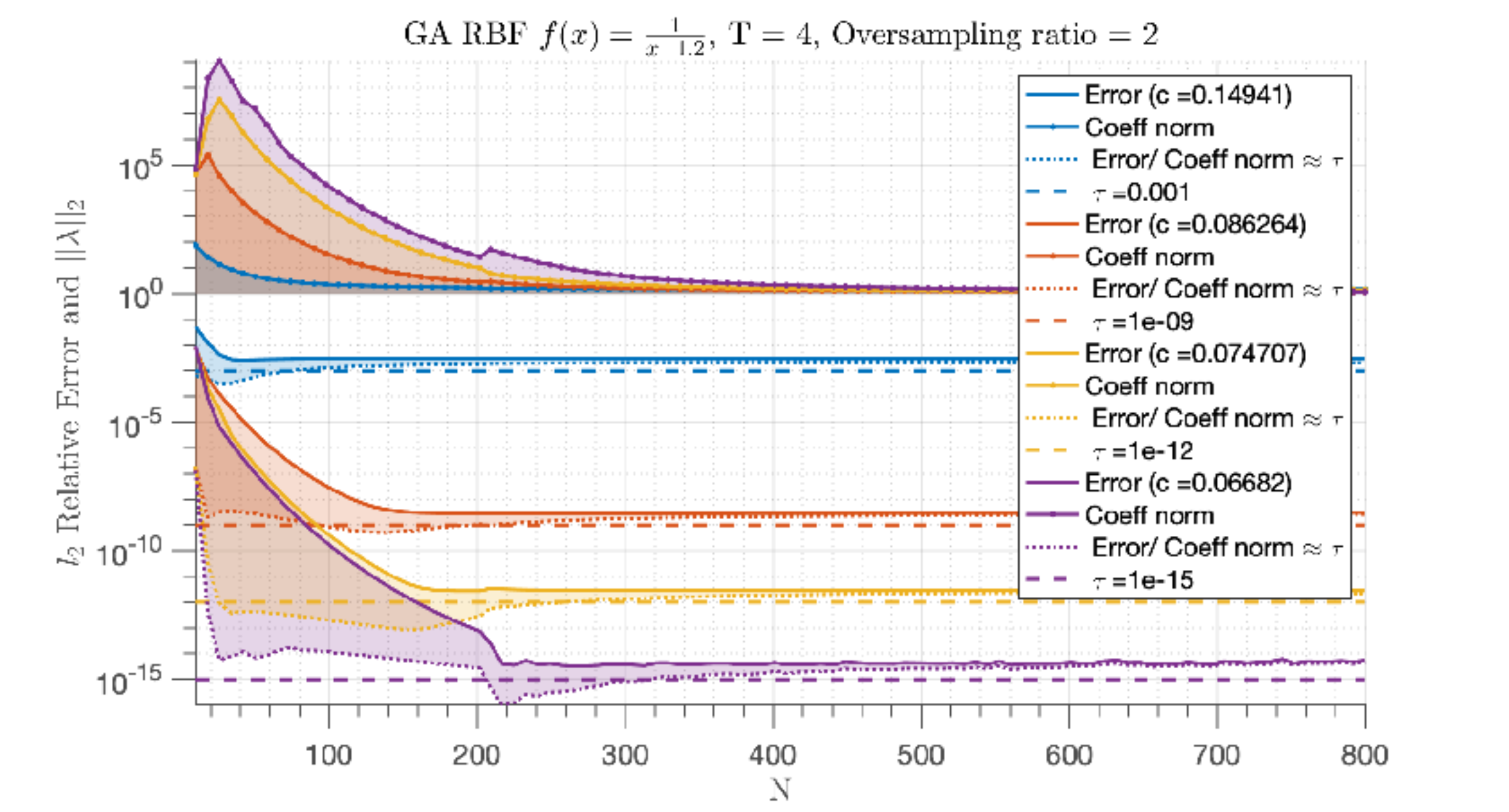}
    \end{minipage}
    \caption{Same as the bottom row of Fig.~\ref{fig:linearscaling_optimalc}, but with an extra indication of the value of $\tau$ and of the ratio of the accuracy over the coefficient norm (dotted lines). As a rule of thumb, this ratio is approximately equal to the threshold $\tau$.}\label{fig:ruleofthumb}
\end{figure}

The rule of thumb is illustrated in Fig.~\ref{fig:ruleofthumb}, where we repeat a part of Fig.~\ref{fig:linearscaling_optimalc} and we have simply added the ratio \eqref{eq:ruleofthumb} to the figure. The dotted lines do approximate $\tau$ in all cases, including in the pre-asymptotic regime. This is an indication that an oversampling factor of $2$ was sufficient for these two examples. Note that the constants $c_{M,N}^\tau, d_{M,N}^\tau$ can in principle also be computed numerically, see~\cite[\S3.6]{fna2}.

\subsection{Two-dimensional examples}
\label{ss:2dcase}

The explicit analytical results of this paper are specific to the univariate Gaussian function. These results can be straightforwardly extended to a tensor product setting. However, that would lead to basis functions aligned with the coordinate axes of the form $\phi_2(x,y) = \phi(x) \phi(y)$. Such functions are rarely used in practice because radial basis functions in 2D are, as the name implies, defined in terms of the radial distance:
\[
 \phi(\mathbf{x},\mathbf{y}) = \phi\left( |\mathbf{x}-\mathbf{y}|\right).
\]
In the next experiments we use the latter, standard definition of radial basis functions. Strictly speaking, the theory of this paper no longer applies. However, we illustrate that the guiding principles remain the same. In particular, we show that the linear scaling regime does lead to accurate approximations using radial basis functions in 2D, and that the choices of the parameters involved correspond to striking a balance between coefficient norm and approximation accuracy.

In an analogous manner to our one-dimensional setup, a bounding box of adjustable size $[-T_1,T_1]\times[-T_2,T_2]$ is chosen to contain the computational domain $\Omega$. The $N$ centers are placed in an hexagonal pattern inside the bounding box. The $M$ sample nodes within $\Omega$ also have a hexagonal pattern, although an investigation of optimal node distribution with possible local node refinement is out of the scope of this paper. The Gaussian radial function is uniquely fit for this setup since it is the only one that is separable in $x$ and $y$; i.e. $e^{-\epsilon^2 (x-x_n)^2+(y-y_n)^2} = e^{-\epsilon^2 (x-x_n)^2} e^{-\epsilon^2 (y-y_n)^2}$. So, since the number of centers in $x$ and in $y$ grow as $O(\sqrt{N})$, we choose $\epsilon = c \sqrt{N}$. An optimal $c$ value for the two-dimensional case is also out of the scope of this study.

\begin{figure}
    \centering
    \begin{minipage}{0.5\textwidth}
        \centering
        \includegraphics[trim=10cm 0cm 10cm 5cm, clip,width=1\textwidth]{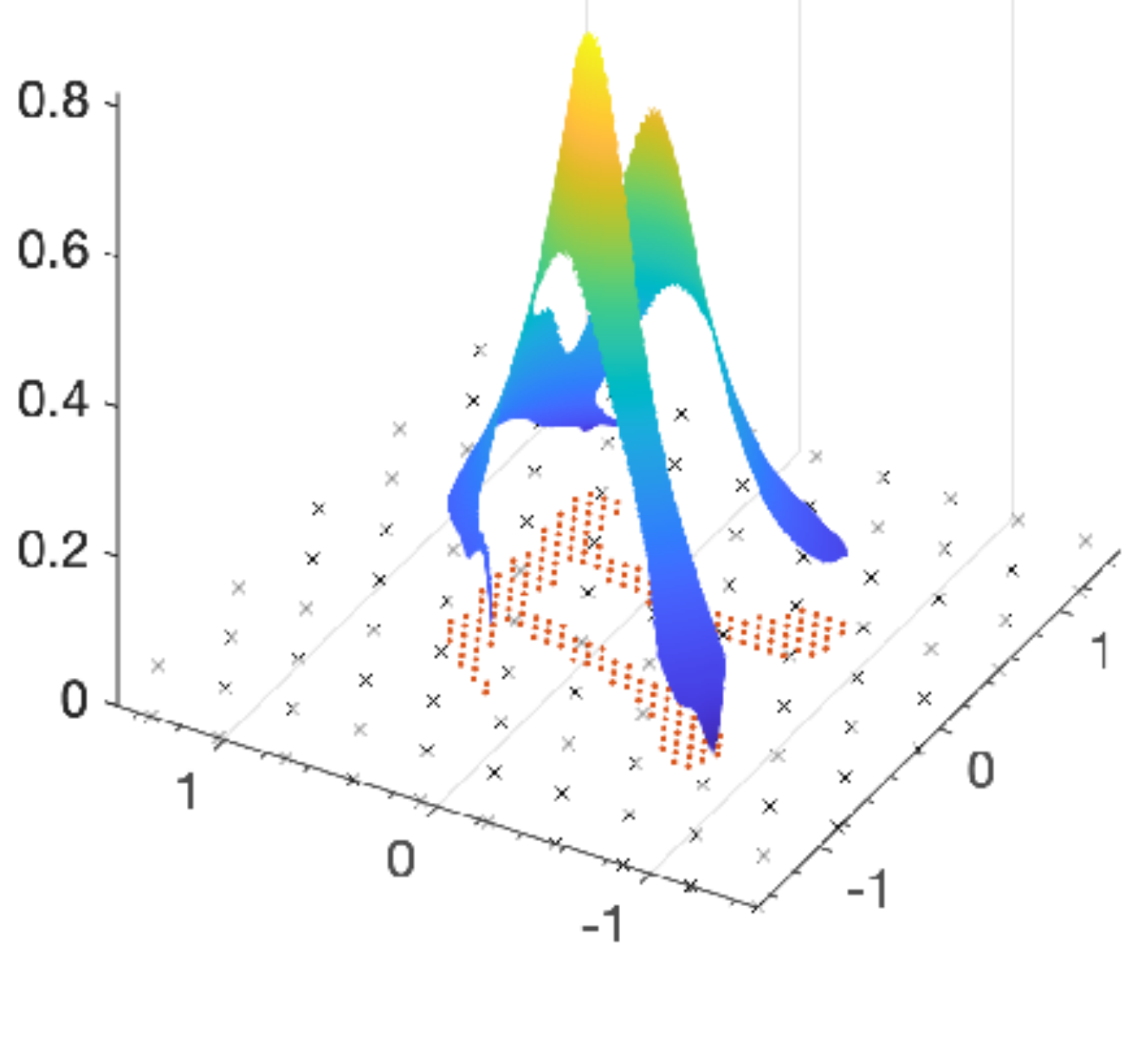}
    \end{minipage}\hfill
    \begin{minipage}{0.5\textwidth}
        \centering
        \includegraphics[trim=7cm 0cm 7cm 0cm, clip,width=1\textwidth]{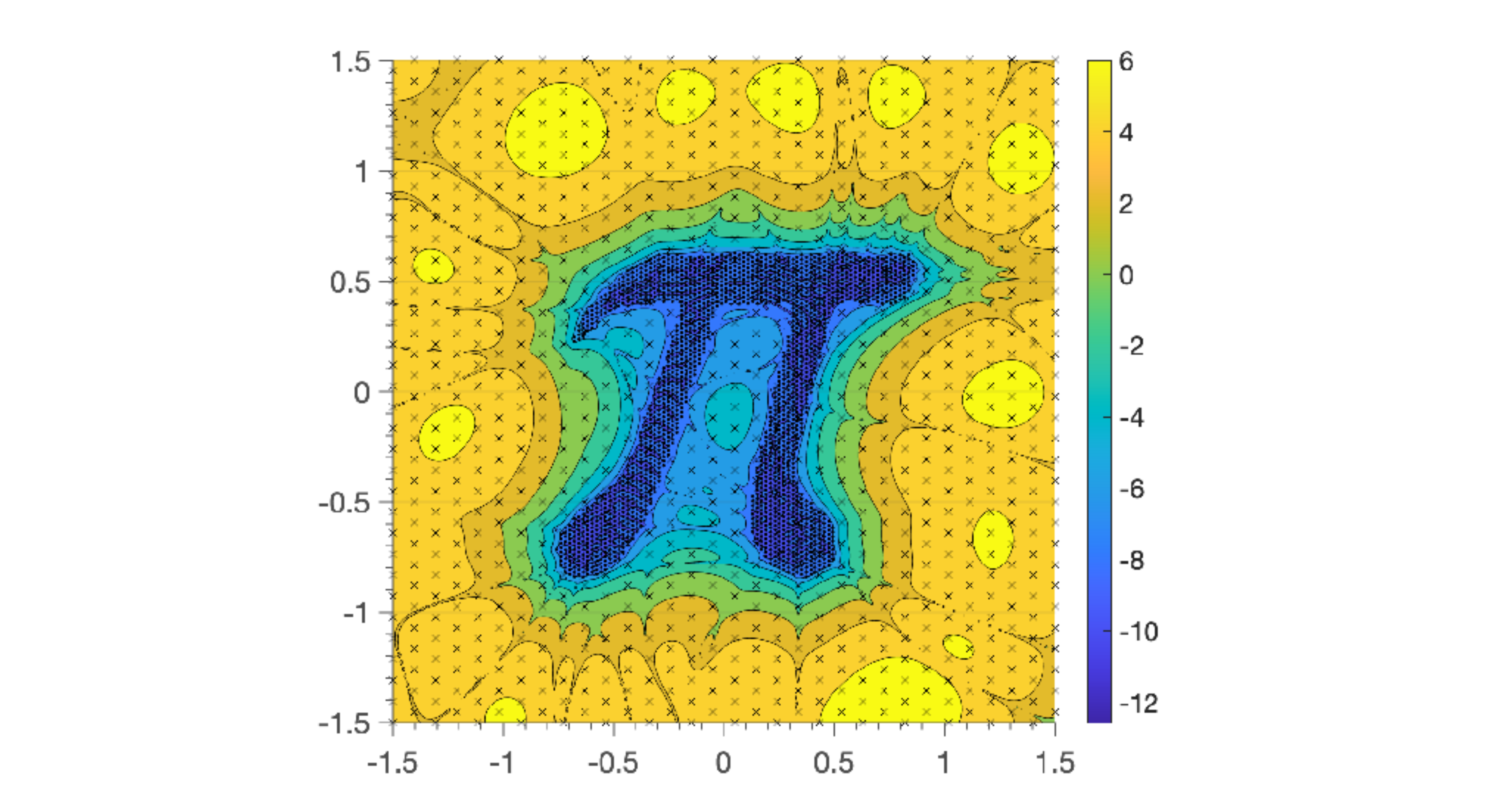}
    \end{minipage}
    \caption{Approximation of function $f(x,y)=\frac{1}{1+10 (x^2+y^2)}$ with $100$ centers and $200$ sample points inside the computational domain on the left and $1,000$ centers and $2,000$ sample points on the right. The right panel shows the approximation error in logarithmic scale. }\label{fig:2DCase}
\end{figure}

An example of this approximation scheme is shown in Fig.~\ref{fig:2DCase}.

\section{A least squares collocation scheme}

We conclude the paper with some experiments to indicate the usefulness of radial basis functions in combination with least squares for the solution of elliptic boundary value problems. To that end we formulate an RBF based collocation scheme, inspired by Kansa's method \cite{KansaMethod1,KansaMethod2},  using more collocation points than degrees of freedom.

We wish to solve $$-\Delta u = f, \; \rm{in} \; \Omega$$ subject to the Dirichlet boundary condition
\[
u = g, \; \rm{on} \; \delta \Omega.
\]
We choose $M_\Omega$ points $\{x_m\}$ inside $\Omega$, and an additional $M_{\partial \Omega}$ points $\{ y_m \}$ on the boundary. This results in a least squares system once $M = M_\Omega + M_{\partial \Omega} > N$.

We set out to solve the linear system $B \lambda = b$, where
\begin{equation*}
\label{eq:PDE_matrix_and_rhs}
B = \begin{pmatrix}
    -\Delta \phi_n (x_m)_{m,n=1}^{M_\Omega,N}\\[0.5cm]
    \phi_n (y_m)_{m,n=1}^{M_{\delta \Omega},N}\\
  \end{pmatrix} \in \mathbb{R}^{M \times N},\; b = \begin{pmatrix}
    f(x_m)^{M_\Omega}_{m=1}\\[0.5cm]
    g(y_m)^{M_{\partial \Omega}}_{m=1}\end{pmatrix} \in \mathbb{R}^{M}.
\end{equation*}
The upper block of the linear system is precisely the least squares matrix for the approximation of $f$ on $\Omega$, using the basis functions $-\Delta \phi_n$. The lower block is the least squares discretization of the boundary condition on $\partial \Omega$.

\begin{figure}[t]
    \centering
    \begin{minipage}{0.46\textwidth}
        \centering
        \includegraphics[width=1\textwidth]{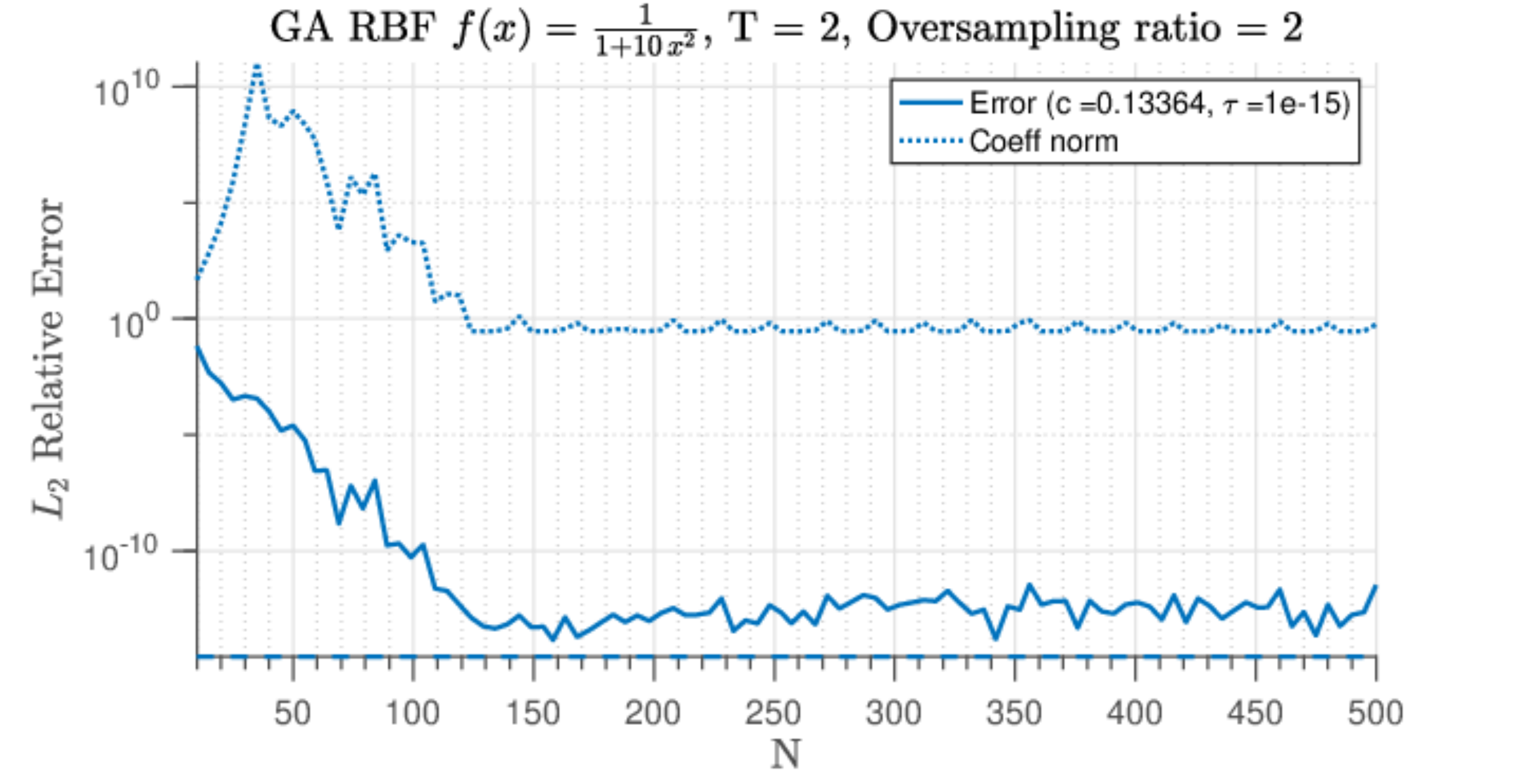}
    \end{minipage}\hfill
    \begin{minipage}{0.54\textwidth}
        \centering
        \includegraphics[width=1\textwidth]{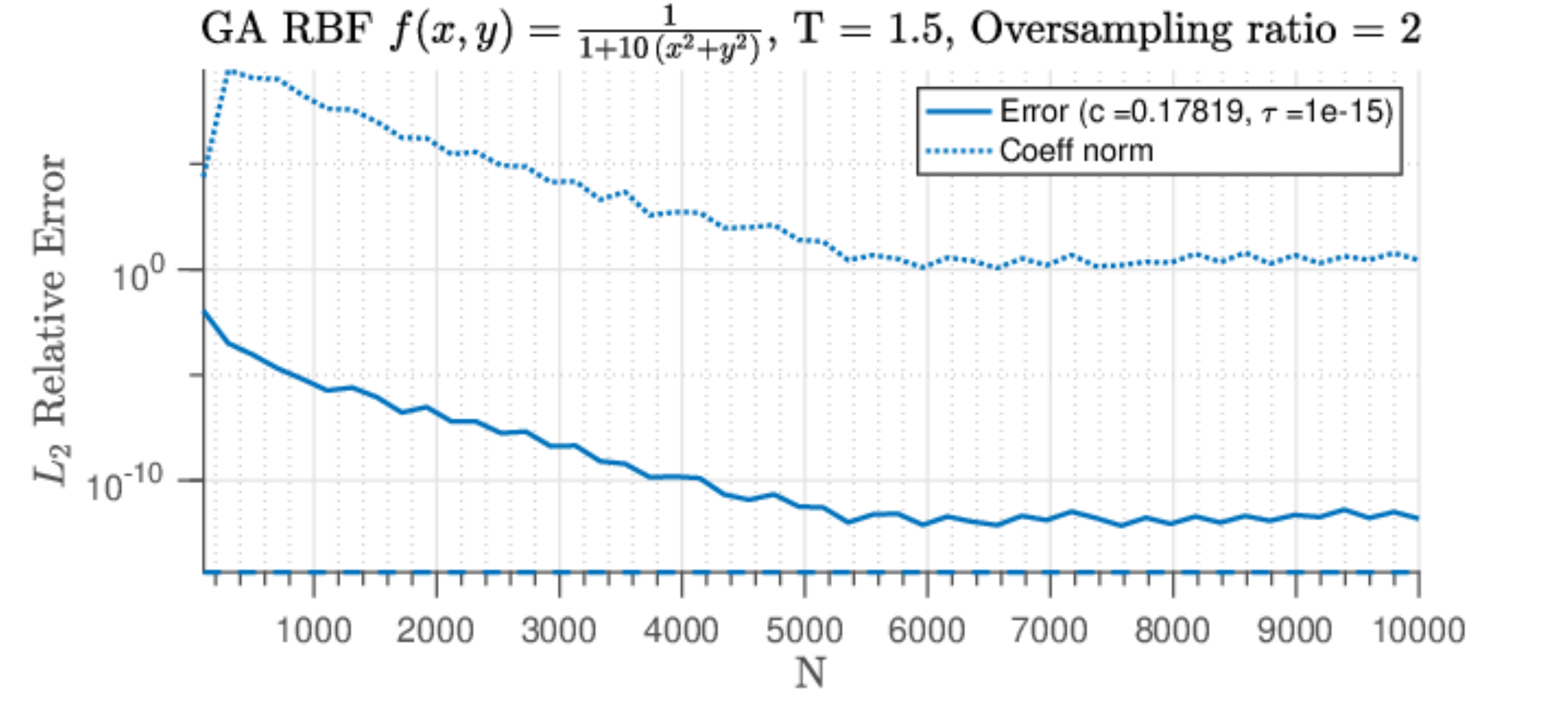}
    \end{minipage}
    \caption{The figure shows approximation error and coefficient norm as a function of the number of degrees of freedom $N$ for the solution of the one-dimensional ODE $-u''(x)=20 \frac{1-30 x^2}{(1+10 x^2)^3}$ with true solution $f(x)=\frac{1}{1+10 x^2}$ (left panels), and the two-dimensional PDE $-\Delta u = 40 \frac{1-10(x^2+y^2)}{(1+10(x^2+y^2))^3}$ with true solution $f(x,y)=\frac{1}{1+10(x^2+y^2)}$ (right panels). We chose the threshold $\tau=1e-15$. In both cases high accuracy is achieved and maintained for larger values of $N$.}\label{fig:PDE}
\end{figure}

The approximate solution is therefore
\begin{equation*}\label{RBF_direct_solution}
 f(x) \approx \sum_{n=1}^N \lambda_n \phi(\varepsilon(x - \xi_n)).
\end{equation*}
Results are shown in Fig.~\ref{fig:PDE} for a one-dimensional and two-dimensional Poisson equation. We have chosen linear scaling $\epsilon = c N$ using the proportionality constant~\eqref{eq:optimal_c} in 1D, and scaling of the form $c \sqrt{N}$ for the 2D problem, exactly as in~\S\ref{ss:2dcase} and similarly using a hexagonal grid.

The domain resembles the number $\pi$ and is based on a smooth parameterization of its boundary using a Fourier series of length $32$. The boundary values are derived from the given analytical solution. Other parameters, such as the number of points $M_\Omega$ and $M_{\partial \Omega}$, have been chosen heuristically to ensure sufficient oversampling.

A noteworthy observation is that the results in Fig.~\ref{fig:PDE} demonstrate high accuracy, which is maintained for larger values of $N$. This contrasts with other methods for PDEs based on global RBF approximations since Kansa, e.g.~\cite[\S3.7]{chen2014rbfcollocation} and~\cite[\S4.1]{fornberg2015rbf_pde}.

%\BLUE{
%TODO: we refer to Kansa's method, perhaps we can show that convergence to machine precision (rather than square root) can be achieved using least squares?
%}

\section{Concluding remarks}
\label{sec:conclusions}

The analysis and numerical results of this paper have confirmed the possibility of accurate and stable approximations using the Gaussian RBF in 1D. This leaves several interesting options for further research, such as the extension of the analysis to 2D and higher-dimensional problems and to other kinds of radial basis functions. An interesting computational challenge is the rapid solution of the rectangular linear systems. Finally, it is clear that in the application of RBF's to partial differential equations much remains to be explored.

%Some conclusions here.

%\subsection*{Open problems}

%\begin{itemize}
%    \item Analytic functions, geometric convergence
%    \item Extend analysis to $d$ dimensions
%    \item Extend analysis to other decaying RBFs, e.g. inverse quadratics
%    \item Global RBFs, e.g. quadratics, polyharmonic,...
%\end{itemize}

\section*{Acknowledgements}

We would like to thank Nick Trefethen for discussions on the topic of this paper, as well as for useful comments on this manuscript. This research was initiated during an inspiring colloquium meeting at Michigan Tech, for which we like to thank the institution.

\bibliographystyle{abbrv}
\bibliography{references}
\end{document}